\documentclass[10pt]{article}

\pdftrailerid{}

\usepackage[utf8]{inputenc}
\usepackage[T1]{fontenc}
\usepackage{amsmath,amssymb,dsfont}
\numberwithin{equation}{section}
\usepackage{microtype}
\usepackage{graphicx,tikz,pgfplots}
\graphicspath{{images/}}
\pgfplotsset{compat=newest}
\usepackage{amsmath}      
\usepackage{subcaption}   
\usepackage{subcaption}
\usepackage{aliascnt}
\usepackage[a4paper,centering,bindingoffset=0cm,marginpar=2cm,margin=2.5cm]{geometry}
\usepackage[pagestyles]{titlesec}
\usepackage[font=footnotesize,format=plain,labelfont=sc,textfont=sl,width=0.75\textwidth,labelsep=period]{caption}
\usepackage{enumitem}
\usepackage{comment}
\usepackage{mathtools}
\usepackage{bigints}
\usepackage{multirow}
\usepackage{makecell}
\usepackage[ruled,vlined]{algorithm2e}
\usepackage{float}
\usepackage{multirow}
\usepackage{amsthm}       
\usepackage{thmtools}     

\usepackage[backend=biber,maxnames=10,backref=true,hyperref=true,giveninits=true,safeinputenc]{biblatex}
\DefineBibliographyStrings{english}{%
	backrefpage = {cited on page},
	backrefpages = {cited on pages},
}

\DeclareFieldFormat[report]{title}{``#1''}
\DeclareFieldFormat[book]{title}{``#1''}
\AtEveryBibitem{\clearfield{url}}
\AtEveryBibitem{\clearfield{note}}

\titleformat{\section}[block]{\large\sc\filcenter}{\thesection.}{0.5ex}{}[]
\titleformat{\subsection}[runin]{\bf}{\thesubsection.}{0.5ex}{}[.]

\usepackage{hyperref}     
\hypersetup
{
    bookmarksnumbered,
    breaklinks,
    pdfborderstyle=,
    colorlinks,
    citecolor=[rgb]{0,.65,0},
    linkcolor=[rgb]{0,0,.5},
    urlcolor=[rgb]{0,0,.5},
    pdfauthor={Authors},
    pdfsubject={Subject},
    pdftitle={Title},
    pdfkeywords={Keywords}
}

\newpagestyle{headers}
{
	\headrule
	\sethead[\footnotesize\thepage][\footnotesize\sc Authors][]{}{\footnotesize\sc Title}{\footnotesize\thepage}
	\setfoot{}{}{}
}
\declaretheorem[numberwithin=section,name=Theorem]{theorem}

\declaretheorem[sibling=theorem]{lemma}

\declaretheorem[sibling=theorem,style=definition]{definition}

\declaretheorem[style=remark,sibling=theorem]{remark}
\declaretheorem[style=remark,sibling=theorem]{example}

\newcommand{\N}{\mathds{N}}

\newcommand{\R}{\mathds{R}}

\let\RE\Re
\let\Re=\undefined
\DeclareMathOperator{\Re}{\RE e}
\let\IM\Im
\let\Im=\undefined
\DeclareMathOperator{\Im}{\IM m}

\newcommand{\abs}[1]{\left|#1\right|}

\newcommand{\set}[1]{\left\{#1\right\}}

\newcommand{\e}{\mathrm e}
\let\ii\i
\renewcommand{\i}{\mathrm i}

\newcommand{\ve}{\varepsilon}

\title{Neural network parametrized level sets for image segmentation}
\author{Otmar Scherzer$^{1,2,3}$\\{\footnotesize\href{mailto:otmar.scherzer@univie.ac.at}{otmar.scherzer@univie.ac.at}}
\and Cong Shi$^{1,2}$\\{\footnotesize\href{mailto:cong.shi@univie.ac.at}{cong.shi@univie.ac.at}}
\and Thi Lan Nhi Vu$^{1,3}$\\{\footnotesize\href{mailto:thi.lan.nhi.vu@univie.ac.at}{thi.lan.nhi.vu@univie.ac.at}}}

\date{}

\newcommand{\vx}{{\vec{x}}}

\newcommand{\vw}{\vec{w}}

\newcommand{\levelfun}{{\tt l}} 
\newcommand{\multiphase}{{\tt m}}
\newcommand{\affine}{{\tt a}} 
\newcommand{\network}{{\tt s}} 
\newcommand{\nettwo}{{\tt t}} 
\newcommand{\nnset}{{\tt S}} 
\newcommand{\nntwo}{{\tt T}} 
\newcommand{\mult}{\imath}

\newcommand{\f}{{\tt f}}

\begin{document}

\maketitle
\thispagestyle{empty}
\begin{center}
\hspace*{5em}
\parbox[t]{12em}{\footnotesize
$^1$Faculty of Mathematics\\
University of Vienna\\
Oskar-Morgenstern-Platz 1\\
A-1090 Vienna, Austria}
\hfil
\parbox[t]{17em}{\footnotesize
\hspace*{-1ex}$^2$Johann Radon Institute for Computational\\
and Applied Mathematics (RICAM)\\
Altenbergerstraße 69\\
A-4040 Linz, Austria}
\end{center}

\begin{center}
\parbox[t]{19em}{\footnotesize
\hspace*{-1ex}$^3$Christian Doppler Laboratory\\
for Mathematical Modeling and Simulation\\
of Next Generations of Ultrasound Devices (MaMSi)\\
Oskar-Morgenstern-Platz 1\\
A-1090 Vienna, Austria}
\end{center}

\let\oldinclude\include
\renewcommand{\include}[1]{\input{#1}}

\begin{abstract}
Chan-Vese algorithms have proven to be a first-class method for image segmentation. Early implementations used level set methods with a pixelwise representation of the level set function. Later, parametrized level set approximations, such as splines, have been studied and computationally developed to improve efficiency. In this paper, we use neural networks as parametrized approximations of level set functions for implementing the Chan-Vese methods. We show that this approach is efficient because of the equivalence between two layer neural networks and polygonal approximations of level set-based segmentations. In turn, this allows the two-layer network architecture to be interpreted as an ansatz function for the approximate minimization of Chan-Vese functionals.
Based on these theory, we extend the classical Chan-Vese algorithm to a data-driven setting, where prior parameters of the network are obtained through unsupervised training on representative image data. These learned parameters encode geometric structures of the data, leading to improved initialization and faster convergence of the Chan–Vese image segmentation.
\end{abstract}

\section{Introduction} \label{sec:intro}
\emph{Image segmentation} aims to partition an image into regions of desired properties. Accurate segmentation plays a crucial role in many areas of applications such as object recognition, medical diagnosis, autonomous driving, and scene understanding, to name but a few.

Basis for this work are the variational Chan-Vese segmentation algorithms \cite{ChaVes01,VesCha02}: In the older paper Chan \& Vese segmented a \emph{single} object, while in the newer they considered segmentation of \emph{multiple} objects. These methods compute piecewise constant functions that minimize a Mumford-Shah like energy functional. The regions of constant intensity of the minimizer represent segments of desired properties. The theoretical foundations of the Chan-Vese model date back to the seminal work of Mumford \& Shah \cite{MumSha89b} and Morel \& Solimini \cite{MorSol95b}. In computational practice, the Chan-Vese functionals are often minimized via a level set approach, where the segments are represented by the level sets of a high dimensional function. Level set methods, originally introduced by Osher \& Sethian \cite{OshSet88}, have proven a powerful approach to solve image segmentation tasks, because the implicit representation of segments allows to identify objects of complex topologies, such as, for instance, segments consisting of multiple connected domains. Building upon this foundational framework, Chung \& Vese \cite{ChuVes05,ChuVes09} extended the Chan-Vese model to a multiphase level set framework capable of simultaneously segmenting multiple objects. Brox \& Weickert \cite{BroWei04} proposed a variational minimization strategy within a level set framework that robustly optimizes both the number of desired segments. Kang et al \cite{KanSanYip11,KanMar14,KanShaSte14} have complemented level set methods which allow for automatic determination of the number of objects.
\cite{KanMar14} proposed models that provide accurate approximations of curve length, where total variation regularization was shown to be the most numerically stable for noisy images. \cite{KanSanYip11} proposed a data driven approach by adapting $k$-means clustering to perform scale segmentation, allowing the method to distinguish objects of different scales rather than different intensities as in the classical setting. In this framework, the number of objects is automatically determined. \cite{KanShaSte14} proposed a supervised data driven approach in which segmentation is learned from labeled points and subsequently applied to the remaining parts of the image or to other similar images.

There exist two categories of level set algorithms (see \autoref{fig:rela}):
\begin{description}
	\item{\emph{Unparametrized}} level set algorithms consider the input images and the level set functions, from which segments and classes are detected, pixel wise. All functions are represented as a vector of pixel values.
	\item{\emph{Parametrized}} algorithms (see, for instance, \cite{YanFucJueSch06,FeiFucJueSchYan08,AghKilMil11,OzsKilStuSaiMil25}) use ansatz-functions (such as splines or finite elements) to represent the level sets and input images. This means, in particular, that the segments and classes are specified \emph{double implicitly}: as the zero level sets of the representing functions (first), which in turn are represented via the parametrization (second).
\end{description}
The different concepts can be compared with \emph{finite difference} and \emph{finite element} methods for solving differential equations: The latter represents the solution of the differential equation implicitly by finite element parametrizations, whereas in finite difference methods the function values of the solution are given as a vector associated with solution values on the grid points. In fact, \emph{hybrid} level set algorithms combining the advantages of both parametrized and unparametrized strategies have also been developed (see \cite{FucJueSchYan09b}).
\begin{figure}[H]
	\centering
	\includegraphics[width=0.5\linewidth]{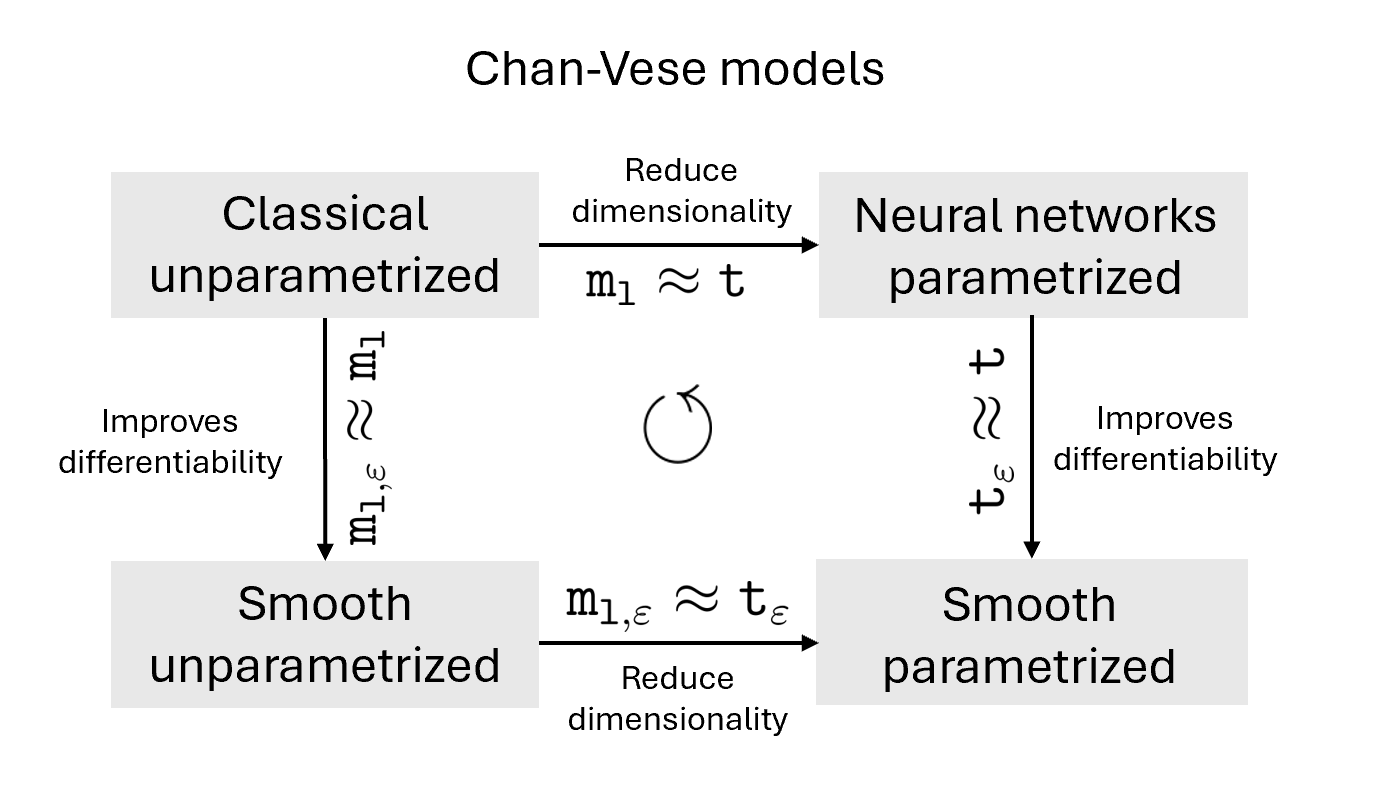}
	\caption{Commutative diagram illustrating the relationships among the classical, parametrized, and smooth Chan-Vese models. The classical (unparametrized) model can be approximated by the parametrized non smooth model, which in turn can be approximated by the smooth variant. Moreover, the smooth parametrized model can also approximate the smooth unparametrized model, which itself provides a smooth approximation of the classical model. The explicit formulations of all Chan–Vese models are given in \autoref{Chan_Vese}, and the corresponding abbreviations are summarized in \autoref{tab:notation}.
    }\label{fig:rela}
\end{figure}

The achievements of this paper are as follows:
\begin{enumerate}
\item We propose a \emph{parametrized} Chan-Vese method for image segmentation and multiple object segmentation, where we use \emph{deep neural networks} for parametrization of the level set functions.
\item The efficiency of neural networks is unquestioned nowadays, but there remains a lack of interpretability regarding why this is the case and how they should be parametrized (this means, which networks one should implement)? We give illustrative examples that demonstrates why (deep) two hidden layer networks efficiently represent segments (see \autoref{LS_NN}) defined from level sets. This result puts the mathematical foundations of level set methods and neural networks in perspective (see \autoref{Chan_Vese}).
\item In contrast to previously documented implementations of Chan-Vese algorithms, the one presented here is \emph{data driven} (see \autoref{subsec:alg} and \autoref{numerics}). In fact, we estimate the \emph{appropriate} neural network structure of a two hidden layer network (which we identified in the previous sections as most adequate): In particular we identify in a data driven manner the adequate number of neurons and prior weights of the neural network parameters.
\end{enumerate}
For the readers convenience we summarize the notation, which is used throughout the paper:
\begin{table}[H]
	\caption{Used notation throughout the paper}
	\centering%
	\begin{tabular}{|c|c|c|}
\hline
\thead{Name} & \thead{Notation} & \thead{Reference} \\  \hline
\thead{level set function}          & $\levelfun$                         & \eqref{eq:GammaI}              \\ \hline
\thead{Heaviside and sigmoid activation function}          & $\sigma$, $\sigma_\ve$                 & \eqref{eq:heavi}              \\ \hline
\thead{Affine linear function}        & $\affine$                           & \eqref{eq:ell}                 \\ \hline
\hline
\thead{One layer (shallow) Heaviside and sigmoid neural network}        & $\network$, $\network_\ve$ 	                & \eqref{eq:n}, \autoref{sec:approx_}                   \\ \hline
\thead{Set of one layer  Heaviside and sigmoid networks}     & $\nnset$, $\nnset_\ve$                            & \eqref{eq:N}, \autoref{sec:approx_}                    \\ \hline
\thead{Two layer Heaviside and sigmoid neural network}           & $\nettwo$, $\nettwo_\ve$                          & \eqref{eq:t}, \autoref{sec:approx_}  \\ \hline
\thead{Set of two layer  Heaviside and sigmoid networks}     & $\nntwo$, $\nntwo_\ve$                        & \eqref{eq:T}, \autoref{sec:approx_}           \\ \hline
\thead{Number of neurons in the first layer}  & $n_1$                      & \eqref{eq:n},  \eqref{eq:t}          \\ \hline
\thead{Number of neurons in the second layer} & $n_2$                      & \eqref{eq:t}   \\ \hline
\hline
\thead{Number of level lines or level sets}         & $m$                                 & \eqref{eq:GammaI0}        \\ \hline
\thead{Piecewise constant function} & $\multiphase$  & \eqref{eq:CV_piecewise} \\ \hline
\thead{Non smooth and smooth multiphase level set function} & $\multiphase_\levelfun$, $\multiphase_{\levelfun,\ve}$  & \eqref{eq:multiphaselevelset}, \autoref{sec:approx_}   \\ \hline
\thead{Non smooth and smooth multiphase function for polygons} & $\multiphase_\network$, $\multiphase_{\network,\ve}$   & \eqref{eq:para}, \autoref{sec:approx_}   \\ \hline
\hline
\thead{Multi-indices of length $m$ with entries in $\{-1,1\}$}       & $\mult$                      & \eqref{eq:multiphaselevelset}, \eqref{eq:GammaI}, \eqref{eq:GammaP}    \\ \hline
\thead{Disjoint regions to be segmented forming a partition of $\Omega$}       & $S_\mult$                      & \eqref{eq:GammaI}             \\\hline
\thead{Polygonal partition of $\Omega$ approximating regions to be segmented}       & $P_\mult$                      & \eqref{eq:GammaP}             \\\hline \hline
\thead{Unparametrized Chan-Vese functional for sets}                & $\mathrm{CV}_{\tt f}$  & \eqref{eq:CV_LS} \\ \hline
\thead{Non smooth and smooth unparametrized}                & \multirow{2}{*}{$\mathrm{CV}^{\mathrm{LS}}_{\tt f}$, $\mathrm{CV}^{\mathrm{LS}}_{\tt f, \ve}$}  & \multirow{2}{*}{\eqref{eq:CV_LS}, \autoref{sec:approx_}} \\
\thead{Chan-Vese functional for multiphase level sets}    &   &   \\ \hline
\thead{Non smooth and smooth parametrized}                  &  \multirow{2}{*}{$\mathrm{CV}^{\mathrm{LS}}_{\tt f}$, $\mathrm{CV}^{\mathrm{LS}}_{\tt f, \ve}$}  & \multirow{2}{*}{\eqref{eq:paraChanVese_l}, \autoref{sec:approx_}} \\
\thead{Chan-Vese functional for multiphase level sets}    &   &   \\ \hline
\end{tabular}\label{tab:notation}
\end{table}

In the following, we start with an example motivating the efficiency of neural networks for segmentation. We then provide an overview of Chan-Vese segmentation models. Throughout this paper, we constrain ourselves to an image domain $\Omega \subseteq \R^2$, meaning that the application we have in mind is multiple object segmentation. Without explicitly stating, we always assume that $\Omega$ is bounded, satisfies the cone property (see \cite{Ada75}), and has piecewise $C^1$-boundary.

\section{Heaviside and sigmoid networks} \label{LS_NN}
In this section, we define Heaviside and sigmoid networks:
\begin{definition}[Heaviside and sigmoid networks] \label{network}
Let $\ve > 0$. The functions
	\begin{equation} \label{eq:heavi}
		\begin{aligned}
			\sigma : \R \to \R\,, \quad
			x \mapsto \left\{ \begin{array}{rcl}
				1 & \text{ for } & x >0\\
				\frac{1}{2} & \text{ for } & x=0\\
				0 & \text{ for } & x <0
			\end{array} \right.	
		\end{aligned} \quad \quad \text{ and }  \quad \quad \begin{aligned}
			\sigma_\ve : \R \to \R\,, \quad
			x \mapsto \frac{1}{1+\e^{-\frac{x}{\ve}}},
		\end{aligned}
	\end{equation}
are called Heaviside  and sigmoid functions, respectively.
	\begin{itemize}
		\item A \emph{Heaviside neuron},
		\begin{equation} \label{eq:HN}
		    N = \sigma \circ {\tt a}\,,
		\end{equation}
		is a composition of an Heaviside and an \emph{affine linear} function
		\begin{equation}\label{eq:ell}
			\begin{aligned}
				{\tt a} : \R^2 \to \R\,, \quad 	\vx \mapsto \vw^T\vx + b,
			\end{aligned}
		\end{equation}
		which is parameterized by $\vw \in \R^2$, $b \in \R$.
		\item A \emph{one layer Heaviside neural network} (often also called \emph{shallow} neural network) (see \autoref{fig:onelayer}) is a linear combination of $n_1$ Heaviside neurons $N_j$, $j=1,\ldots,n_1$:
		\begin{equation}\label{eq:n}
			\vx \in \R^2 \mapsto \network(\vx) := \network[\vec{a},{\bf W},\vec{b}](\vx) := \sum_{j=1}^{n_1} a_j N_j(\vx) = \sum_{j=1}^{n_1} a_j (\sigma \circ {\tt a}_j)(\vx) = \sum_{j=1}^{n_1} a_j \sigma (\vw_j^T\vx + b_j)\,,
		\end{equation}
		where ${\bf W} = (\vw_1,\ldots,\vw_{n_1}) \in \R^{2 \times n_1}$, $\vec{a}= (a_1,\ldots,a_{n_1}) \in \R^{n_1}$ and $\vec{b} = (b_1,\ldots,b_{n_1})\in \R^{n_1}$ are called the parametrizations of the function $\network$. The set of all one layer  Heaviside neural networks is denoted by
		\begin{equation} \label{eq:N}
			\nnset := \set{\network[\vec{a},{\bf W},\vec{b}]: {\bf W} \in \R^{2 \times n_1}, \vec{a} \in \R^{n_1}, \vec{b} \in \R^{n_1}}\;.
		\end{equation}
		\item A \emph{two layer Heaviside neural network with $n_1$ neurons in the first layer and $n_2$ neurons in the second layer} (see \autoref{fig:twolayer}) is defined as
		\begin{equation}\label{eq:t}
			\vx \in \R^2 \mapsto \nettwo(\vx) : = \nettwo[{\bf A},\vec{c}, {\bf D}, \vec{b}, {\bf W}](\vx) := \sum_{i=1}^{n_2} c_i \sigma\left( \sum_{j=1}^{n_1} a_{ij} (\sigma \circ {\tt a}_j)(\vx) + d_{ij} \right)\;.
		\end{equation}
		${\bf A} = (a_{ij}) \in \R^{n_2 \times n_1}$, $\vec{c} =(c_i) \in \R^{n_2}$, ${\bf D} = (d_{ij}) \in \R^{n_2 \times n_1}$ and $\vec{b}=(b_j)\in \R^{n_1}$, ${\bf W} = (\vw_j) \in \R^{2 \times n_1}$ (note that affine linear function ${\tt a}_j$ (see \autoref{eq:ell}) depends on ${\bf W}$ and $\vec{b}$) are the parametrization of ${\tt t}$. The set of all two layer Heaviside neural networks is denoted by
		\begin{equation} \label{eq:T}
			\nntwo: = \set{ \nettwo[{\bf A},\vec{c}, {\bf D}, \vec{b}, {\bf W}] :  {\bf A} \in \R^{n_2 \times n_1}, \vec{c} \in \R^{n_2}, {\bf D} \in \R^{n_2 \times n_1}, \vec{b}\in \R^{n_1}, {\bf W} \in \R^{2 \times n_1}}\,.
		\end{equation}
		Note the functions ${\tt a}_j$ do not depend on the index $i$ of the second layer.
		\item A \emph{sigmoid neural network} is defined analogously to a Heaviside neural network by replacing the Heaviside function $\sigma$ with a sigmoid function $\sigma_\ve$. Accordingly, we denote one layer  and two layer sigmoid neural network by $\network_\ve$, ${\tt t}_\ve$, respectively, and the set of one layer  and two layer sigmoid neural networks by $\nnset_\ve$ and $\nntwo_\ve$, respectively.
	\end{itemize}
\end{definition}
We have schematically visualized the one layer  and two layer neural networks in \autoref{fig:NN}.

\begin{figure}[H]
\begin{subfigure}[t]{0.5\textwidth}
	\centering
	\includegraphics[width=1\linewidth]{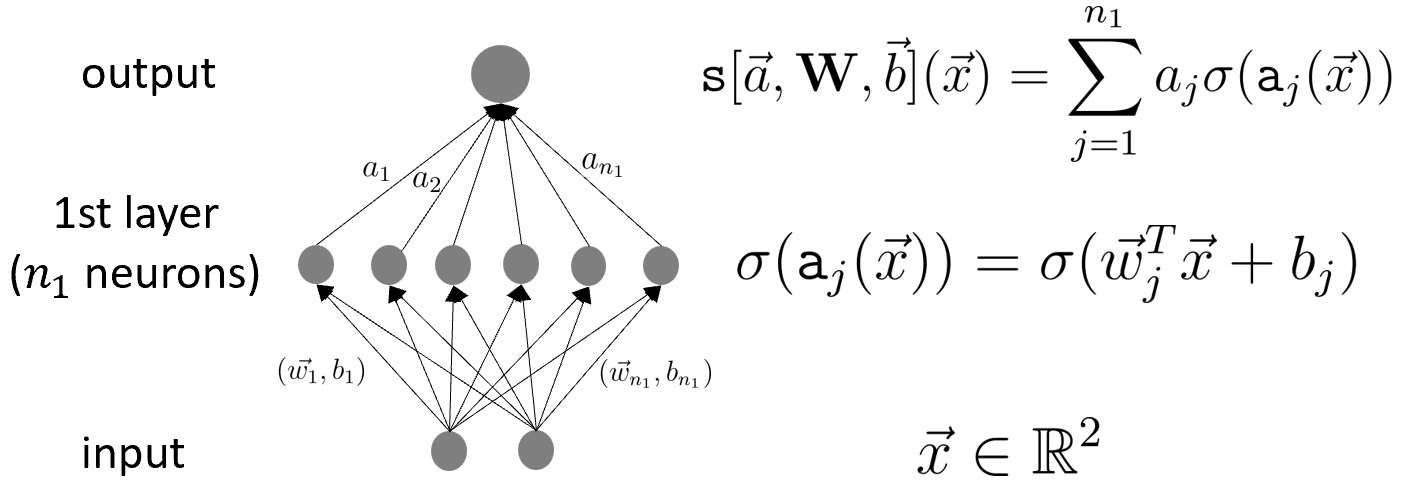}
	\caption{A one layer  neural network is used to approximate level set functions.}
	\label{fig:onelayer}
\end{subfigure}%
~
\begin{subfigure}[t]{0.5\textwidth}
	\centering
	\includegraphics[width=1\linewidth]{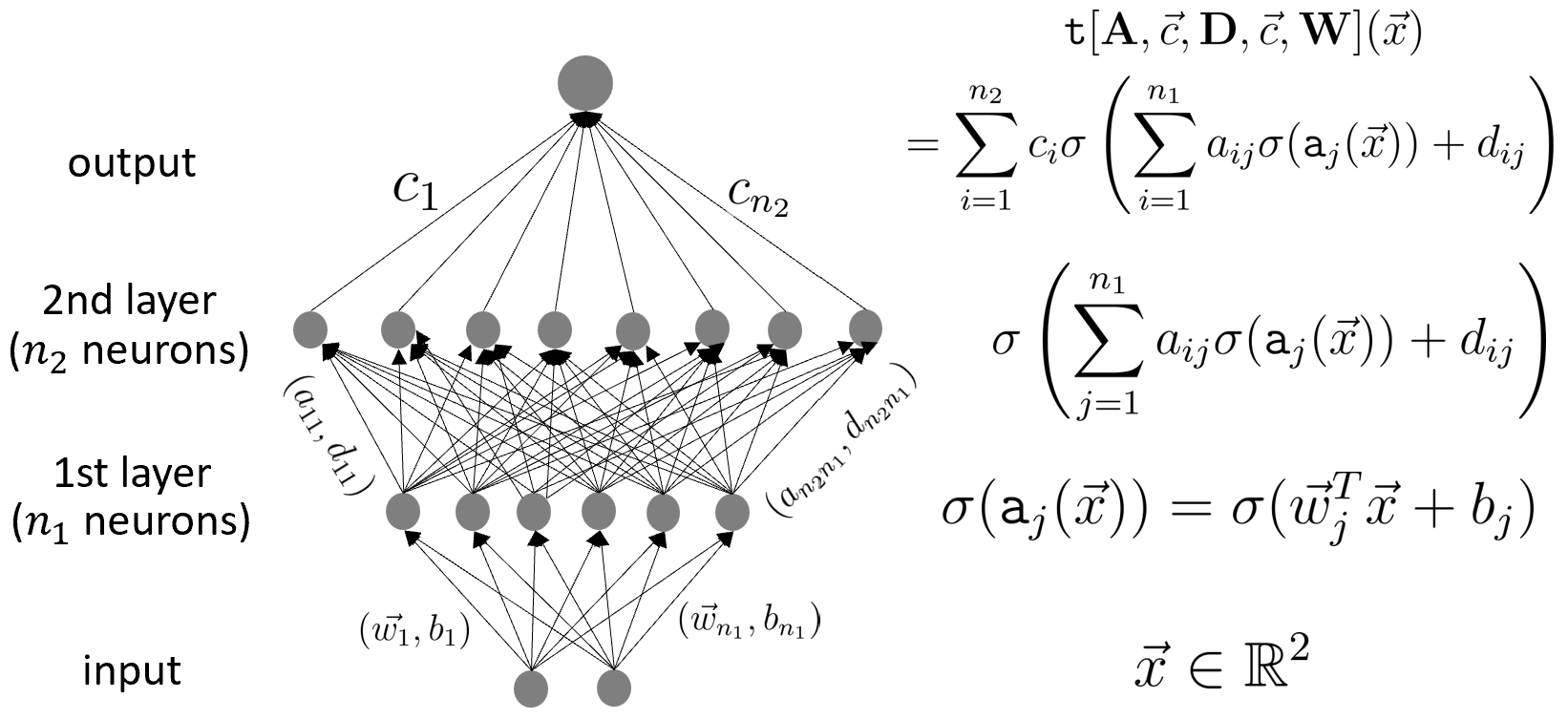}
	\caption{Two layer neural network are used to approximate (multi)-phase functions.}
	\label{fig:twolayer}
\end{subfigure}
\caption{One layer and two layer neural networks.}
\label{fig:NN}
\end{figure}
$\nnset$ (see \autoref{eq:N}) and $\nntwo$ (see \autoref{eq:T}) are sets of neural networks. These sets are not minimal, meaning that different combinations of parameters can potentially represent a given function. In practice, some parameters are fixed to reduce computational complexity. The term ``customized'' refers to the situation that parameters from the general formulation \autoref{eq:n} and \autoref{eq:t} are fixed a-priori. We add a subscript $c$ in the functions to emphasize on the customization.
\begin{definition}[Customized two layer network] \label{de:cust}
The network
	\begin{equation}\label{eq:d}
		\begin{aligned}
			\nettwo_c=\nettwo_c[\vec{a},{\bf W},\vec{b},d]: \R^2  \to \R\,, \quad
			\vx \mapsto \sigma \left( \network[\vec{a},{\bf W},\vec{b}](\vx) + d \right)= \sigma \left(\sum_{j=1}^{n_1} a_j \sigma (\vw_j^T\vx + b_j) + d\right)\,,
		\end{aligned}
	\end{equation}
\end{definition}
is called a \emph{customized} two layer network (see \autoref{eq:t}) with $n_2=1$, $c_1=1$, $a_{ij} = a_{1j} =: a_j$, and $d_{ij}=d_{1j}=d$. In this network architecture, the second layer of $t_c$ consists of a single neuron, so only the first layer of parameters is trainable.

The network $\nettwo_c$ in \autoref{eq:d} is specifically used for polygon parametrization (see \autoref{sec:example}).

\subsection{Heaviside networks for parametrizations of polygons} \label{sec:example}\hspace{0pt}\\

In the following, we consider the segmentation of a polygon in $\R^2$, first illustrated with the example of a triangle (see \autoref{fig:motiv_eg_2}).

	\begin{definition}[Parametrized polygons] \label{de:triangle} Let $3 \leq n_1 \in \N$ and let 
		${\tt a}_j$, $j=1,2,\ldots,n_1$ affine linear functions as defined in \autoref{eq:ell} with $\vw_j \neq 0$. The three sets associated with ${\tt a}_j$, $j=1,2,\ldots,n_1$, 
		\begin{equation} \label{eq:halfplane}
			H_j^1 := \set{\vx: {\tt a}_j(\vx) > 0},\, H_j^0 := \set{\vx: {\tt a}_j(\vx) = 0} \text{ and } H_j^{-1} := \set{\vx: {\tt a}_j(\vx) < 0}\,,
		\end{equation}
		are called \emph{parametrized half space decomposition} and the sets
		\begin{equation}\label{eq:omegajota}
			\Omega_\mult := \bigcap_{j=1}^{n_1} H_j^{\mult_j}, \quad \text{ where } \mult = (\mult_1,\mult_2,\ldots,\mult_{n_1}) \in \set{-1,0,1}^{n_1}\,,
		\end{equation}
		are called \emph{parametrized polygons}.
\end{definition}

\begin{remark} \label{le:general} We see from \autoref{eq:omegajota} that if there exists some index $j \in \set{1,\ldots,n-1}$ such that $\mult_j=0$, then $\Omega_\mult$ is a subset of the line of the boundary of the polygon. It can even be a point or an empty set.
\end{remark}

\begin{example}[Heaviside network of a triangle] \label{ex:triangle} We consider three affine linear functions ${\tt a}_j$, $j=1,2,3$, as defined in \autoref{eq:ell} with $\vw_j \neq 0$, where the level lines are in general position.
Assume that all 3 lines are in general position, meaning that no two lines are parallel and no three lines intersect at a single point, so that the combinatorial structure of the intersections is maximally generic.
The triangle is the intersection of the sets $H_j^1$, $j=1,2,3$, such that the parametrizations $\vec{w}_j$, $j=1,2,3$ is the normal vector of the zero level set of $\affine_j$ pointing into the triangle.

There exist seven non-empty open sets as defined in \autoref{eq:omegajota}.
    Note that $\Omega_{(-1,-1,-1)} = \emptyset$ because if it would not be empty all zero level sets $H_j^0$, $j=1,2,3$ meet in one point, which cannot happen because we assume that the triangle is not degenerate.
    Now we define a customized one layer  Heaviside network (compare \autoref{eq:n}), where the number of neurons equals the number of lines, that is, $n_1=3$:
    \begin{equation} \label{eq:nc}
    	\network_c(\vx) := \sum_{j=1}^3 \frac{1}{3} \sigma({\tt a}_j(\vx)) \text{ for all } \vx \in \R^2\;.
    \end{equation}
    $\network_c$ attains the values $1/3,2/3,1$ almost everywhere. However, we can also make a finer analysis of $\network_c$ by considering the following sets of measure zero
    \begin{equation}\label{eq:omegajotaII}
    	\Omega_\mult := \bigcap_{j=1}^3 H_j^{\mult_j}, \quad \text{ where } \mult = (\mult_1,\mult_2,\mult_3) \in \set{-1,0,1}^3 \text{ and some } \mult_j=0\;.
    \end{equation}
    The latter combinations \autoref{eq:omegajotaII} can only happen at zero level set $H_j^0$, $j=1,\ldots,3$.
    Moreover, $\Omega_{(0,0,0)} = \emptyset$ because if it would not be empty all zero level sets $H_j^0$, $j=1,2,3$ meet in one point. Moreover, the sets $\Omega_{(0,0,-1)}$ and $\Omega_{(0,-1,-1)}$ are empty (see \autoref{fig:triang_inter} and take into account \autoref{le:general}).
    We consider different cases:
    \begin{enumerate}
    	\item At a point $\vx$ where $\iota =(1,1,1)$ we have $\network_c(\vx)=1$. At a point where $\iota=(1,1,0)$ we have$\network_c(\vx)=2/3$ and where $\iota=(1,0,0)$ we have $\network_c(\vx)=1$.
    	\item For \begin{equation*}
    			\iota = (1,0,0), (1,1,0), (1,0,-1)
    		\end{equation*}
    		we get the
    		\begin{equation}\label{eq:values}
    			\network_c(\vx)=2/3,5/6,1/2\,,
    		\end{equation}
    		respectively. $\Omega_{(1,1,0)}$ and $\Omega_{(1,0,0)}$ are edges and vertices of the triangle, respectively and $\Omega_{(1,0,-1)}$ correspond to the edges of the complement polygons that do not share boundaries with the triangle (see \autoref{fig:triang_inter}).
    \end{enumerate}

	\begin{figure}[H]
		\centering   \includegraphics[width=1\linewidth]{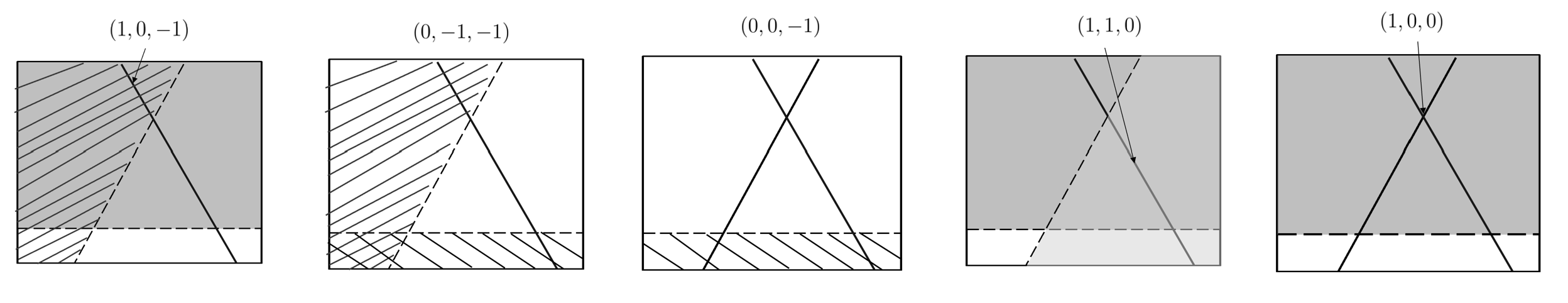}
		\caption{Possible intersections of $\cap_{j=1}^3 H_j^{\iota_j}$ with at least one $\iota_j = 0$. The gray regions represent $H_j^1$, dashed regions represent $H_j^{-1}$ and solid line represents $H_j^{0}$. The intersections give either edges of the polygons in the complement of the triangle (first image); or empty sets (second and third images); or edges and vertices of the to-be-segmented triangle (fourth and last images).}
		\label{fig:triang_inter}
	\end{figure}	
	
For a choice
\begin{equation} \label{eq:kappa_ptw}
	-1 < \kappa < -1+\frac{1}{2n_1}= -5/6,
\end{equation}
the two layer Heaviside network (compare with \autoref{eq:d})
\begin{equation} \label{eq:nettwocp}
   \nettwo_c^+(\vx) := \sigma \left( \kappa + \network_c (\vx) \right)
\end{equation}
coincides pointwise with $\chi_{\Omega_{(1,1,1)}}$. In a neural network analogy, $\sigma({\tt a}_j)$, $j=1,2,3$, corresponds to an activation of three neurons ${\tt a}_j$ in the first layer. The functions $\network_c$ and $\nettwo_c^+$ denote customized one layer  and two layer neural network, respectively. Some coefficients in the general form of a neural network have been customized, such as  $a_j=1/3$, for $j=1,2,3$, and $d = \kappa$.

\begin{figure}[H]
	\centering   \includegraphics[width=1\linewidth]{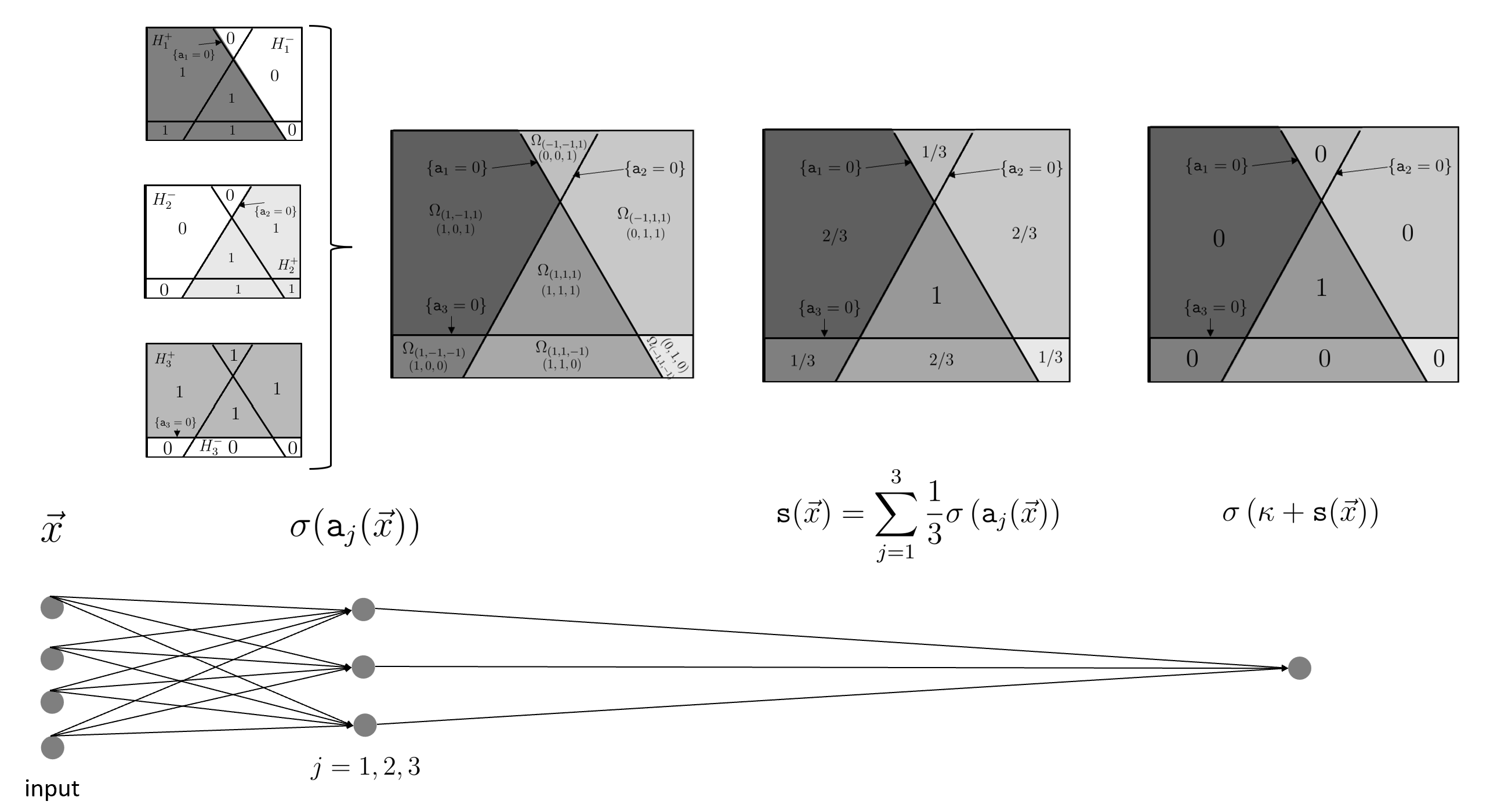}
	\caption{\textbf{1.} Each line $H_j^0 = \{{\tt a}_j = 0\}$, for $j=1,2,3$, divides the image domain into 2 half planes $H_j^1$ and $H_j^{-1}$. The half planes containing the triangle $T$ (shown in gray) attain the value $1$ and the complementary half planes (shown in white) attain the value $0$. The three images each depict a binary partition (gray and white regions), corresponding to the activation of a single neuron in the first layer of a neural network. Taken together, these binary regions represent the activations of three neurons in the first layer, represented by $\sigma({\tt a}_j)$, $j = 1, 2, 3$, where ${\tt a}_j > 0$ inside the gray region and ${\tt a}_j(\vx) < 0$ outside the gray region, as in the standard level set formulation.\\
		\textbf{2.} Seven subregions $\Omega_{\iota}$ for $\iota \in (-1,1)^3$ and their triplets of activation values $(\sigma({\tt a}_1), \sigma({\tt a}_2), \sigma({\tt a}_3))$.\\
		\textbf{3.} The average of these activations triplets are taken over each subregion $\Omega_\iota$, resulting in values $1/3,2/3,1$. This corresponds to the linear combination $\network_c = \sum_{j=1}^3 \frac{1}{3}\sigma({\tt a}_j)$.\\
		\textbf{4.} The function $\sigma(\kappa + \network_c)$ attains the value $1$ on the triangle $\Omega_{(1,1,1)}$. On the complement of $\Omega_{(1,1,1)}$, $\sigma(\kappa + \network_c)$ is zero almost everywhere.}\label{fig:motiv_eg_2}
\end{figure}
\end{example}

While in the previous example we considered representing the characteristic function of a triangle with a customized two layer network, we now consider representing a piecewise constant function on the triangle $\Omega_{(1,1,1)}$ and its complement.
\begin{example}
	The two layer network
\begin{equation} \label{eq:CVfctpara}
   \nettwo(\vx) = c_1\nettwo_c^+(\vx) + c_{-1}\nettwo_c^-(\vx)= c_1\chi_{\Omega_{(1,1,1)}}(\vx) + c_{-1}\chi_{\Omega\backslash\Omega_{(1,1,1)}}(\vx)\,,
\end{equation}
where
\begin{equation*}
\nettwo_c^{-}(\vx) := 1-\nettwo_c^+(\vx) =1  - \sigma \left( \kappa+\network_c(\vx)\right) = \sigma \left( -\kappa - \network_c(\vx)\right) = \chi_{\Omega\backslash\Omega_{(1,1,1)}}(\vx)\,,
\end{equation*}
is piecewise constant on $\Omega_{(1,1,1)}$ and its complement.
$\nettwo$ is a two layer network (see \autoref{eq:t} and \autoref{eq:T}) with coefficients
\begin{equation*}
	\begin{aligned}
		{\bf A} &= (a_{ij}) \in \R^{2 \times 3} \text{ is independent of } j \text{ with } a_{ij} = \frac{(-1)^{i-1}}{3}\,, \\
		\vec{c} &= (c_1, c_{-1}) \in \R^2,\\
		{\bf D} &= (d_{ij}) \in \R^{2 \times 3} \text{ is independent of } j \text{ with } d_{ij} = (-1)^{i-1}\kappa\,,\\ \vec{b} &\in \R^3,\\
		{\bf W} &\in \R^{3 \times 2},  \mathrm{det}([\vw_j,\vw_k])\neq 0 \text{ for all } j,k=1,2,3, j \neq k.
	\end{aligned}
\end{equation*}
\end{example}

Now, we extend the triangle example to more general polygons. It is a question from combinatorics, how many non empty sets exist for general polygons with $n_1$ corners.  The answer is provided in \autoref{le:polygon}. Interestingly the number of non empty sets is not $2^{n_1}-1$ as might be expected from the calculations with the triangle:

\begin{theorem}[\cite{Ste26}] \label{le:polygon} 
	Let $3 \le n_1 \in \N$ be the number of affine linear functions $\affine_j$ with $\vw_j \neq 0$ for $j=1,\ldots,n_1$, and let $\mult = (\mult_1,\ldots,\mult_{n_1}) \in \set{-1,1}^{n_1}$. Let $H_j^{\mult_j}$ be defined as in \autoref{eq:halfplane}. 
	Then, there exist at most $\frac{n_1(n_1+1)}{2} + 1$ non empty sets $\bigcap_{j=1}^{n_1} H_j^{\mult_j}$, and
	$\bigcap_{j=1}^{n_1} H_j^{1}$ is a $n_1$-edge polygon.
\end{theorem}
Note that for $n_1=3$ (that means a triangle) we get $\frac{n_1(n_1+1)}{2} + 1 = 2^{n_1}-1=7$ non empty sets, which coincides with the calculations from \autoref{ex:triangle}. For $n_1 > 3$ the majority of intersection of $H_j^{\mult_j}$ (see \autoref{eq:omegajota}) is however empty.

The motivating \autoref{ex:triangle} demonstrates the efficiency of using customized two layer neural networks (see \autoref{eq:d}) to represent the characteristic function of a polygon (see \autoref{le:polygon}). Since polygons approximate arbitrary objects (see \autoref{fig:parametrized}) it can be anticipated that neural network level sets are efficient parametrizations for level sets in Chan-Vese segmentation methods, which are reviewed in the later sections.
\begin{figure}[H]
	\centering
	\includegraphics[width=0.6\linewidth]{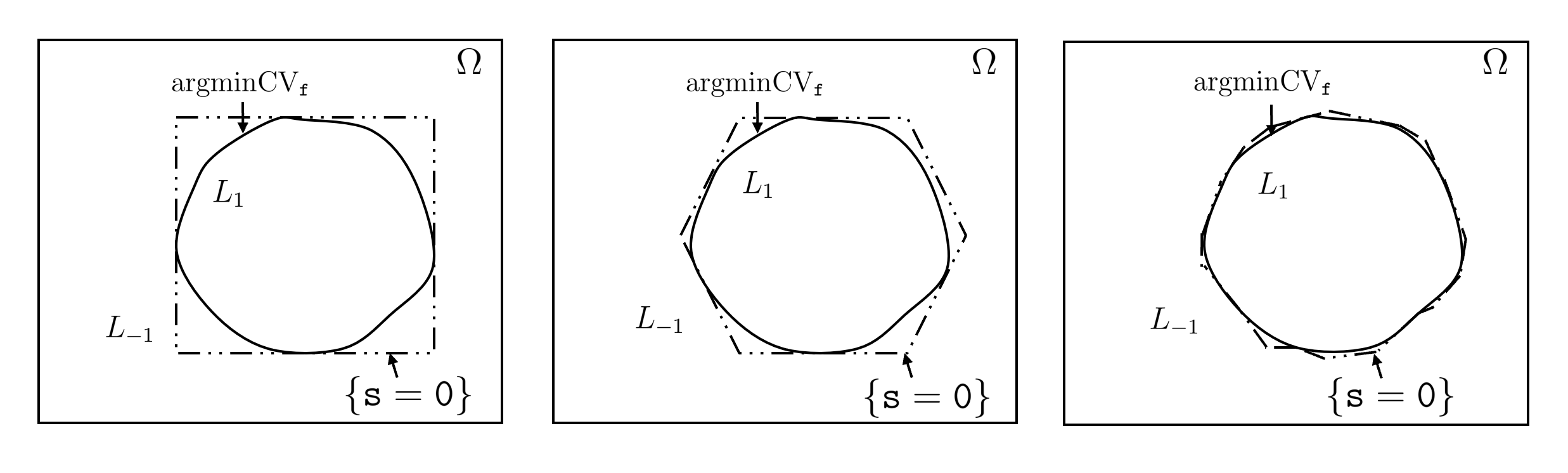}
	\caption{Polygonal approximation of arbitrary bounded regions. Increasing the number of polygonal edges, which corresponds to increasing the number of neurons $n_1$ in the network, improves the approximation accuracy.}\label{fig:parametrized}
\end{figure}
Moreover, for efficient numerical implementation of Chan-Vese methods for segmentation, we approximate the Heaviside activation function $\sigma$ by a smooth sigmoid function $\sigma_\ve$.
This approximation is discussed in \autoref{sec:approx_}. 

\section{Chan-Vese segmentation} \label{Chan_Vese}

We begin by reviewing Chan-Vese segmentation models (see \cite{ChaVes01,VesCha02}) and the according level set methods for computational implementation.

\subsection{The Chan-Vese functional for sets}
We present the general formulation dating back to the multi-object segmentation model from Chan-Vese in \cite{VesCha02}:
Let $m \in \N$ be given and let $\Omega \subseteq \R^2$, which is partitioned into $2^m$ pairwise disjoint open subsets up to their boundaries (see \autoref{fig:intersection} and \autoref{fig:clustering}):
\begin{equation*}
	\set{L_\mult : \mult \in \set{-1,1}^m} \text{ satisfying } \left(\bigcup_{\mult \in \set{-1,1}^m} L_\mult \right) \cup \left(\bigcup_{\mult \in \set{-1,1}^m} \partial L_\mult \right)  = \Omega\;.
\end{equation*}

\begin{figure}[H]
    \centering
    \includegraphics[width=0.8\linewidth]{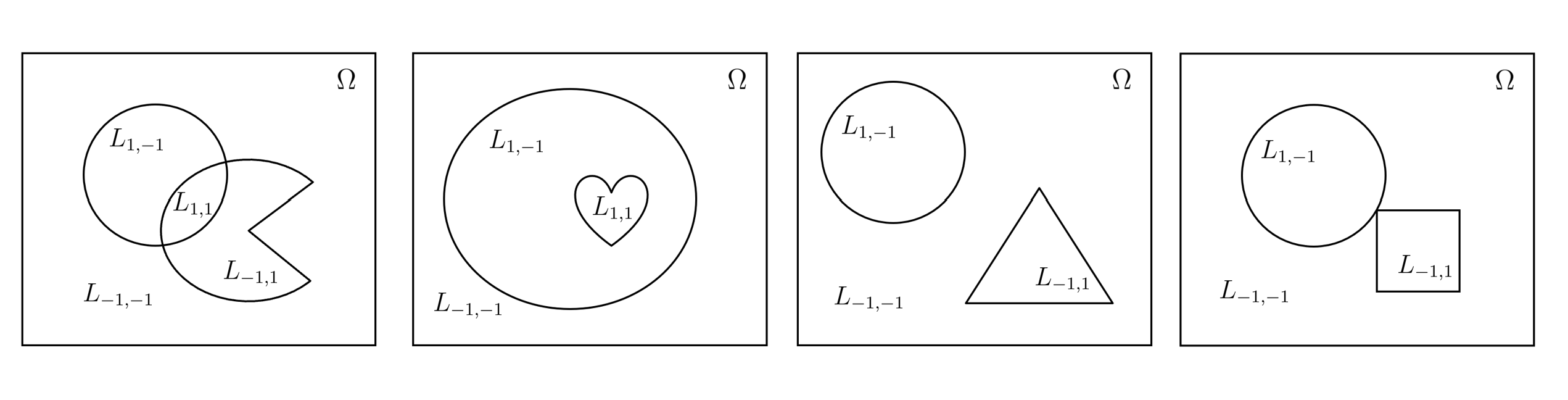}
\caption{Different types of intersections of two objects that induced four disjoint segmented regions $L_{\pm 1, \pm 1}$ (some of these regions may be empty). Each of the object can be described by the positive values of a level set function. \textbf{First:} The intersection has positive Lebesgue measure, but neither region is contained in the other. \textbf{Second:} One region is contained within the other. \textbf{Third:} The two regions are disjoint. \textbf{Fourth:} The two regions intersect only at a finite number of points.}\label{fig:intersection}
\end{figure}
\begin{remark}
When $m=1$, the index set of $\mult$ is simply the set $\set{-1,1}$ and $L_1$ corresponds to a segment and $L_{-1}$ corresponds to the region outside the segment. Analogously, when $m=2$, the index set of $\mult$ contains four elements $\set{(1,1), (1,-1), (-1,1), (-1,-1)}$, which represent four different segments.
\end{remark}

The Chan-Vese model is designed to approximate a given function $\f:\Omega \to \R$
by a piecewise constant
\begin{equation}\label{eq:CV_piecewise}
    \multiphase(\vx) = \sum_{\mult \in \{-1,1\}^m} c_\mult \chi_{L_\mult}(\vx)\;
\end{equation}
which minimizes the Chan-Vese functional
\begin{equation}\label{eq:CV_og}
	\mathrm{CV}_{\tt f}(\{(c_\mult,L_\mult)\}_{\mult \in\{-1,1\}^m}) = \sum_{\mult \in \{-1,1\}^m}\int_{L_\mult} \left( \multiphase - \f \right)^2 \mathrm{d}\vx + \mu \hspace*{-0.03\textwidth} \sum_{\mult \in \{-1,1\}^m} \mathcal{H}^1(\partial L_\mult) + \nu \mathcal{H}^2\left(\bigcup_{\mult \in \{-1,1\}^m} L_\mult\right),
\end{equation}
where $\mathcal{H}^n$, for $n=1,2$, denotes the $n$-dimensional Hausdorff measure. We note that the $2$-dimensional Lebesgue measure coincides for Lebesgue measurable sets with the $2$-dimensional Hausdorff measure on $\R^2$ \cite[Theorem 2.5]{EvaGar15}.

\begin{remark} \label{rmk:c1minimizer}
Since the one-dimensional measures of all $\partial L_\mult$ are penalized, each $L_\mult$ must be a set of finite perimeter (a Caccioppoli set), meaning that $\partial L_\mult$ is measurable with finite one-dimensional measure. Moreover, the main theorem of \cite[Chapter 5]{MorSol95b} guarantees the existence of a segmentation $K := \{\partial L_\mult\}_{\mult \in\{-1,1\}^m}$ that minimizes the functional $\mathrm{CV}_{\tt f}$. The minimizer is not unique (see \cite[Chapter 15]{MorSol95b}). Furthermore, the theorem states that each $\partial L_\mult$ of this minimizer is either a regular $C^1$-curve, or a piecewise regular $C^1$-curve with singular points consisting of triple junctions, where three branches meet at $120^o$ angles, or boundary points where $K$ meets $\partial\Omega$ at $90^o$ angles. This suggests that it is sufficient to study piecewise $C^1$-level sets function for minimization.
\end{remark}

In the following, we first study the classical numerical Chan–Vese method, then introduce its neural network parametrization, and finally present its smooth approximation (as in \autoref{fig:rela}).

\subsection{The Chan-Vese functional for multiphase level sets} \label{sec:parametrized}\hspace{0pt}\\
Chan \& Vese \cite{ChaVes01,VesCha02} proposed a level set formulation of the functional defined in \autoref{eq:CV_og} to minimize it numerically. The formulation utilizes that a piecewise constant function $\multiphase$ can be represented as a \emph{multiphase level set function} (see \autoref{eq:CVfctpara}):
\begin{definition}[\textbf{Multiphase level set function}] \label{lsf}
	Let $m \in \N$. For every $k=1,\ldots,m$, let $\levelfun_k \in C^0(\overline{\Omega})$ be a continuous \emph{level set function}, which in particular implies that
		\begin{align}\label{eq:GammaI0}
			\mathcal{L}_k^1 = \set{\vx \in \Omega : \levelfun_k(\vx) > 0}\,,
			\mathcal{L}_k^{-1} = \set{\vx \in \Omega : \levelfun_k(\vx) < 0} \text{ and } \\
			\partial \mathcal{L}_k^1 = \mathcal{L}_k^0 =  \{\vx \in \Omega : \levelfun_k(\vx) = 0\}, \text{ with }  \mathcal{H}^1(\partial \mathcal{L}_k^1) < \infty. \nonumber
		\end{align}
 \end{definition}
 $\mathcal{L}_k^{\pm 1}$ represents two segments. In the language of level set methods, \autoref{eq:GammaI0} means that the $0$-level set is not fat. We emphasize that a signed distance function of a domain is a typical example of a level set function.

We note that the level set function $\levelfun$ is here an arbitrary continuous function, whereas in \autoref{ex:triangle}, the level set function is ${\tt a}$, which is affine linear. As a result, in \autoref{eq:GammaI0}, the sets $\mathcal{L}_k^{\pm1}$ are bounded open sets and the sets $\mathcal{L}_k^0$ are curves, whereas in \autoref{ex:triangle}, $H_k^{\pm1}$ are half planes and $H_k^{0}$ are lines.

\begin{definition}[\textbf{Multiphase level set function of degree $2^m$}] \label{lsfm}
	Let $m \in \N$. A multiphase level set function \emph{of degree }$2^m$ is a function of the form
	\begin{equation}\label{eq:multiphaselevelset}
		\begin{aligned}
			\multiphase_\levelfun(\vx) = \hspace*{-0.03\textwidth} \sum_{\mult = (\mult_1,\ldots,\mult_m) \in \{-1,1\}^m} \hspace*{-0.02\textwidth} c_\mult \prod_{k=1}^m \sigma(\mult_k \levelfun_k(\vx))\;.		
		\end{aligned}
	\end{equation}
\end{definition}

Since we assumed that $\Omega$ is bounded, $\multiphase_\levelfun \in L^2(\Omega)$ is piecewise constant and takes values $c_\mult \in \R$ on the interior of at most $2^m$ regions $L_\mult$, $\mult \in \set{-1,1}^m$,
where
\begin{align}\label{eq:GammaI}
L_\mult = L_{(\mult_1,\ldots,\mult_m)} = \bigcap_{k=1}^m \set{\vx \in \Omega : \mult_k \levelfun_k(\vx) > 0} = \bigcap_{k=1}^m \mathcal{L}_k^{\mult_k},
\end{align}
where $\mathcal{L}_k^{\pm 1}$ are defined in \autoref{eq:GammaI0}.
The set $\{L_\mult:\mult \in \{-1,1\}^m\}$ forms a partition of $\Omega$ up to sets of measure zero.
Moreover, because of \autoref{eq:GammaI0}, the boundaries of $\{L_\mult:\mult \in \set{-1,1}^m\}$ are the zero level sets of the functions $\{\levelfun_k : k=1,\ldots,m\}$ and we have:
\begin{equation*}
\text{For every } \mult \in \{-1,1\}^m\text{, if } \vx \in \partial L_\mult \text{, then there exists } k \in \{1,\ldots,m\} \text{ such that } \vx \in \mathcal{L}_k^0,
\end{equation*}
and conversely,
\begin{equation*}
\text{for every } k \in \{1,\ldots,m\}\text{ and } \vx \in \mathcal{L}_k^0 \text{, then there exists } \mult \in \{-1,1\}^m \text{ such that } \vx \in \partial L_\mult,
\end{equation*}
where $\mathcal{L}_k^{0}$ are defined in \autoref{eq:GammaI0}. \autoref{fig:clustering} illustrates the set $\{L_\mult:\mult \in \set{-1,1}^m\}$ for $m=1, 2$.

\begin{figure}[H]
	\centering
	\includegraphics[width=0.5\linewidth]{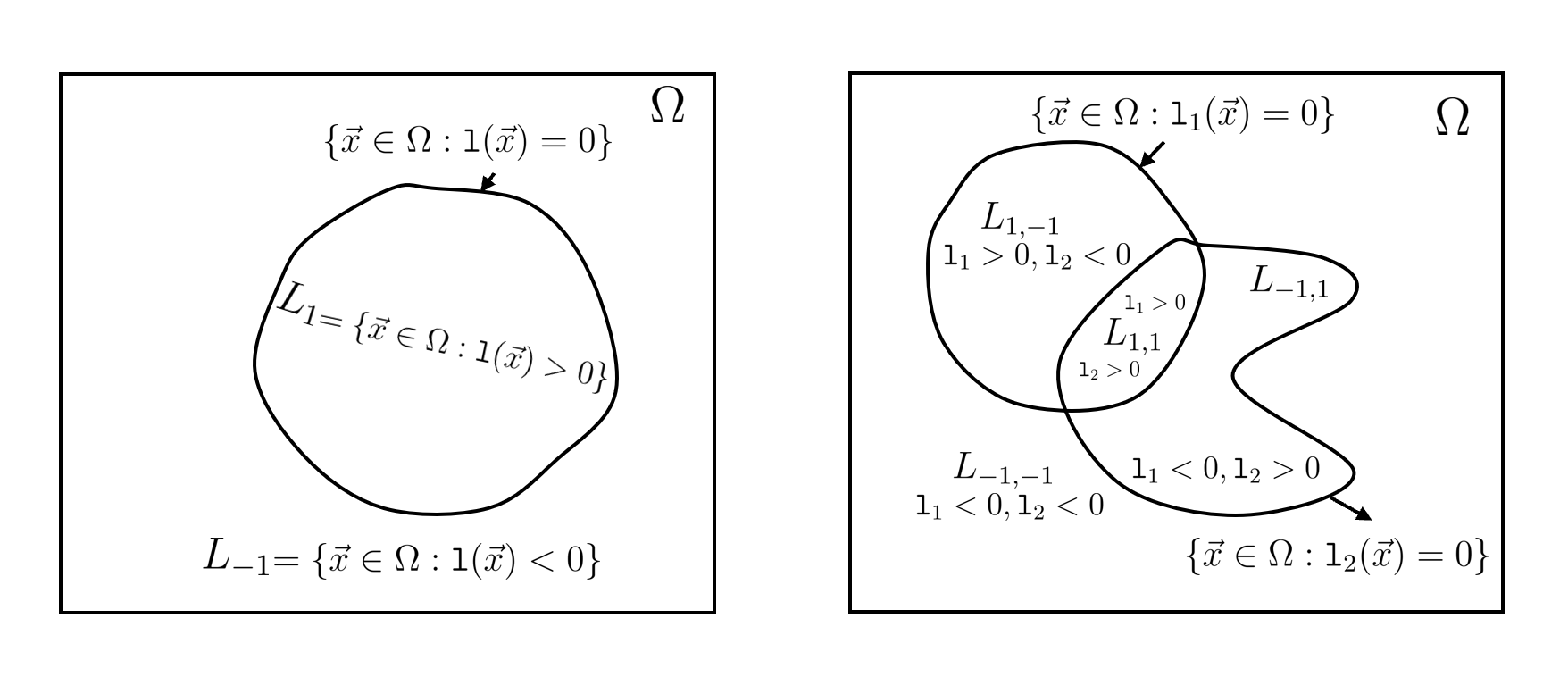}
	\caption{Image segmentation characterized by level set functions and their signs. Here, $1$ denotes the positive sign of a level set function, corresponding to the inside of a region, and $-1$ denotes the negative sign, corresponding to the outside. \textbf{Left:} Segmentation with one level set function $\levelfun$ $(m=1)$, $\Omega$ is segmented in two segments $L_{1}$ and $L_{-1}$. \textbf{Right:} Segmentation with two level set functions $\levelfun_1, \levelfun_2$ $(m=2)$, $\Omega$ to segment four sets $L_{\pm 1,\pm 1}$.}\label{fig:clustering}
\end{figure}

\begin{remark}
	According to \autoref{lsf}, if we aim to identify $s$ segments, we need at least $m =\lceil \log_2 s\rceil$ level set functions, where $\lceil \cdot \rceil$ denotes the smallest integer greater than or equal to its argument.
\end{remark}

The Chan-Vese functional, defined in \autoref{eq:CV_og}, can be reformulated in terms of multiphase level set function $\multiphase_\levelfun$ (see \autoref{lsfm}) as follows:
\begin{equation}\label{eq:CV_LS}
\begin{aligned}
	& \mathrm{CV}^{\mathrm{LS}}_{\tt f}((c_\mult)_{\mult \in \set{-1,1}^m}, (\levelfun_k)_{k=1}^m)\\
	:= &  \sum_{\mult \in \{-1,1\}^m}\int_{\Omega} \left( c_\mult - \f(\vx) \right)^2\prod_{k=1}^m \sigma(\mult_k\levelfun_k(\vx)) \ \mathrm{d}\vx \\
    & \quad + \mu \sum_{k = 1}^m \mathcal{H}^1(\{\levelfun_k = 0\}) 
    + \nu \mathcal{H}^2\left(\bigcup_{j=1}^m \{\levelfun_k > 0\}\right)\\
	= & \sum_{\mult \in \{-1,1\}^m}\int_{\Omega} \left( c_\mult - \f(\vx) \right)^2\prod_{k=1}^m \sigma(\mult_k\levelfun_k(\vx)) \ \mathrm{d}\vx\\
    & \quad + \mu \sum_{k=1}^m \int_{\Omega}\delta(\levelfun_k(\vx))\abs{\nabla\levelfun_k(\vx)} \ \mathrm{d}\vx + \nu \left(\sum_{k=1}^m (-1)^{k+1} \sum_{\substack{A \subseteq \{1,2, \ldots, m\} \\ \abs{A} = k}} \int_{\Omega} \prod_{k \in A}\sigma(\levelfun_k(\vx)) \ \mathrm{d}\vx\right).
\end{aligned}
\end{equation}
Note that the length term $\delta(\levelfun_k)\abs{\nabla\levelfun_k}$ is understood in a distributional sense: it defines a one-dimensional measure supported on the zero level set of $\levelfun_k$, integrating along the normal direction to the curve. Here, the one-dimensional Dirac-mass $\delta$ is the distributional derivative of the Heaviside function $\sigma$.

\begin{example} For $m=1$ and $m=2$, the Chan-Vese functional looks as follows:
\begin{enumerate}
    \item $m=1$:
\begin{equation*}
	\begin{aligned}
		\mathrm{CV}^{\mathrm{LS}}_{\tt f} (c_1, c_{-1}, \levelfun) =& \int_{\Omega} (\f(\vx)-c_1)^{2}\sigma(\levelfun(\vx)) \mathrm{d}\vx + \int_{\Omega} (\f(\vx)-c_{-1})^{2}(1-\sigma(\levelfun(\vx))) \mathrm{d}\vx \\
		& \quad + \mu \int_{\Omega} \delta(\levelfun)\abs{\nabla\levelfun}\mathrm{d}\vx +\nu \int_{\Omega} \sigma(\levelfun(\vx))\mathrm{d}\vx.
	\end{aligned}
\end{equation*} This is used for single object segmentation applications.
    \item $m=2$:
\begin{equation*}
	\begin{aligned}
~ & \mathrm{CV}^{\mathrm{LS}}_{\tt f} ((c_\mult), \levelfun_1,\levelfun_2) \\
= &
\int_{\Omega} (\f(\vx)-c_{1,1})^{2}\sigma(\levelfun_1(\vx))\sigma(\levelfun_2(\vx)) \mathrm{d}\vx + \int_{\Omega} (\f(\vx)-c_{1,-1})^{2}\sigma(\levelfun_1(\vx))(1-\sigma(\levelfun_2(\vx))) \mathrm{d}\vx\\
+ &\int_{\Omega} (\f(\vx)-c_{-1,1})^{2}(1-\sigma(\levelfun_1(\vx)))\sigma(\levelfun_2(\vx)) \mathrm{d}\vx +
  \int_{\Omega} (\f(\vx)-c_{-1,-1})^{2}(1-\sigma(\levelfun_1(\vx)))(1-\sigma(\levelfun_2(\vx))) \mathrm{d}\vx\\
+ &\mu \int_{\Omega} \delta(\levelfun_1)\abs{\nabla\levelfun_1} + \delta(\levelfun_2)\abs{\nabla\levelfun_2} \mathrm{d}\vx +\nu \int_{\Omega} \sigma(\levelfun_1(\vx)) + \sigma(\levelfun_2(\vx)) - \sigma(\levelfun_1(\vx))\sigma(\levelfun_2(\vx))\mathrm{d}\vx\;.
   \end{aligned}
\end{equation*}
Note that we have $\sigma(-x) = 1-\sigma(x)$ for all $x\in \R$, therefore $(1-\sigma(\levelfun_i)) = \sigma(-\levelfun_i)$ for $i=1,2$.
\end{enumerate}
\end{example}

Vese \& Chan \cite{VesCha02} suggested an evolution process, in which $(\levelfun_{k}(t,\vx))_{k=1}^m$ and $(c_\mult(t))_{\mult \in \{-1,1\}^m}$ evolve with time $t$ and, as $t \to \infty$, approximate the minimizers $(\levelfun_k)_{k=1}^m$ of $\mathrm{CV}^{\mathrm{LS}}_{\tt f}$ defined in \autoref{eq:CV_LS} and the corresponding constants $(c_\mult)_{\mult \in \{-1,1\}^m}$. Concretely, starting from some initial level sets $\set{\levelfun_{k,0} = 0 : k=1,\ldots,m}$ and from the initial constant values $c_\mult(0)$ on the level sets, satisfying,
\begin{equation*}
\levelfun_k(0,\vx) = \levelfun_{k,0}(\vx) \text{ for } \vx \in \Omega \text{ and } c_\mult(0) = \frac{\int_{\Omega} \f(\vx)\prod_{k=1}^{m}\sigma(\mult_k\levelfun_{k,0}(\vx)) \mathrm{d}\vx}{\int_{\Omega} \prod_{k=1}^{m}\sigma(\mult_k\levelfun_{k,0}(\vx)) \mathrm{d}\vx}\,,
\end{equation*}
the evolution process is governed by the system
\begin{equation} \label{eq:evo}
\begin{cases} \hspace{0.5cm}
	\begin{aligned}
        \frac{\partial\levelfun_k}{\partial t}(t,\vx) = \delta(\levelfun_k)(t,\vx)
    & \Biggl[
    \mu \mathrm{div}\left(\frac{\nabla\levelfun_k}{\abs{\nabla\levelfun_k}} \right)(t,\vx) \Biggr. \\
    & - \nu \sum_{j=1}^m (-1)^{j+1} \sum_{\substack{A \subseteq \{1,2, \ldots, m\} \\ \abs{A} = j}} \prod_{\substack{i \in A\\ i \neq j}} \sigma(\levelfun_i)(t,\vx)  \\
    & \Biggl.  -  \sum_{\mult \in\{-1,1\}^m}  \mult_k (c_\mult(t) - \f(\vx))^2
    \Biggr] \text{ for } (t,\vx) \in (0,\infty) \times \Omega, \\
    & \hspace{-3.5cm} \text{ together with boundary conditions} \\
    \displaystyle\frac{\delta(\levelfun_k)(t,\vx)}{\abs{\nabla\levelfun_k(t,\vx)}}
        \frac{\partial\levelfun_k}{\partial\Vec{n}}(t,\vx) &= 0  \ \text{ for } (t,\vx) \in (0,\infty) \times \partial\Omega,\\
       & \hspace{-3.5cm} \text{ and the evolution for }  c_\mult:\\
    c_\mult(t) &= \frac{\int_{\Omega} \f(\vx)\prod_{k=1}^{m}\sigma(\mult_k\levelfun_k(t,\vx)) \mathrm{d}\vx}{\int_{\Omega} \prod_{k=1}^{m}\sigma(\mult_k\levelfun_k(t,\vx)) \mathrm{d}\vx} \text{ for } t \in (0,\infty)\;.
\end{aligned}
\end{cases}
\end{equation}
Here, each $\levelfun_k$, $k=1,\ldots,m$, evolves in the direction normal to its level-lines.

\bigskip\par\noindent
In contrast to the standard pixel based formulation (unparametrized) of the Chan-Vese functional in \autoref{eq:CV_LS},
we consider below a \emph{parametrized} Chan-Vese model, where neural networks are used to represent the level set functions $\levelfun$ and the piecewise constant segmentation functions $\multiphase$.

\subsection{Chan-Vese functional for two layer neural networks} \label{subsec:paraCV}\hspace{0pt}\\
We consider the Vese-Chan segmentation of multiple objects, where the level set functions $\set{\levelfun_k}_{k=1}^m$ (see \autoref{eq:multiphaselevelset}) are parametrized by one layer  Heaviside networks. Since polygons approximate arbitrary sets, one layer  Heaviside networks are an adequate ansatz for numerical minimization of the Chan-Vese level set functional.

\begin{definition}[\textbf{Multiphase function for polygons}] \label{lsfm1}
	Let $m \in \N$. A multiphase Heaviside function for polygons is a function of the form
	\begin{equation} \label{eq:para}
	\vx \in \Omega \subseteq \R^2 \mapsto \multiphase_\network(\vx):=\sum_{\mult \in  \{-1,1\}^m} c_\mult \prod_{k=1}^m \sigma(\mult_k\network_k (\vx)).\;
	\end{equation}
\end{definition}

In comparison with \autoref{eq:multiphaselevelset}, we consider here $\sigma$ composed with a one layer  network
	\begin{equation} \label{eq:neto}
		\network_k  := \network[\vec{a}_k,{\bf W}_k,\vec{b}_k] := \sum_{j=1}^{n_1} (\vec{a}_k)_j \sigma ((\vw_{k})_j^T\vx + (\vec{b}_k)_j)\,,
	\end{equation}
	which is defined in \autoref{eq:n} with $(\vec{a}_k)_j, (\vec{b}_k)_j \in \R$ denoting the $j$-th elements of the vectors $\vec{a}_k, \vec{b}_k \in \R^{n_1}$, respectively, and  $(\vw_{k})_j \in \R^2$ is the $j$-th column of the matrix ${\bf W}_k := \left((\vw_{k})_1 \ (\vw_{k})_2 \cdots (\vec{w}_{k})_{n_1}\right) \in \R^{2 \times n_1}$. Note that, in general, the number of neurons $n_1$ can be different for each $k = 1,\ldots,m$. For simplicity of notation, we assume that all one layer  networks $\network_k$, $k=1,\ldots,m$ have the same number of neurons.

\begin{remark}
We recall that $\multiphase_\network$ is defined via one layer networks $\{\network_k : k=1, \ldots, m\}$, while
$\multiphase_\levelfun$, which is defined via general level set functions $\{\levelfun_k : k=1,\ldots,m\}$. However, $\multiphase_\network$ itself is not a one layer  network. Due to the term $\sigma(\mult_k\network_k (\vx))$, $\multiphase_\network$ consists of two layers. Indeed, \autoref{thrm:lsnn} below shows that $\multiphase_\network$ coincide pointwise with a two layer network as defined in \autoref{eq:t}.
\end{remark}

\begin{remark}
	The shifted one layer Heaviside neural network
	\begin{equation} \label{eq:netshift}
		\vx \mapsto \network(\vx) = \sum_{j=1}^{n_1} a_j \sigma (\vw_j^T\vx + b_j) + d
	\end{equation}
	can be written as a one layer Heaviside neural network $\sum_{j=1}^{n_1+1} a_j \sigma (\vw_j^T\vx + b_j)$.
This can be seen by defining 
        $a_{n_1+1} = d$, $\vw_{n_1+1} \in \R^2$ and $b_{n_1+1} \in \R$ in such a way that
$$\sigma(\vw_{n_1+1}^T\vx + b_{n_1+1}) = 1 \text{  for all  } \vx \in \Omega.$$

For example, choose any $\vw_{n_1+1} \in \R^2$ and then select $b_{n_1+1} \in \R$ sufficiently large so that $b_{n_1+1} > \max_{\vx \in \Omega} \abs{\vw_{n_1+1}^T\vx}$,
ensuring that
 $\vw_{n_1+1}^T\vx + b_{n_1+1} \geq - \max_{\vx \in \Omega} \abs{\vw_{n_1+1}^T\vx} + b_{n_1+1} > 0 \text{ for all } \vx \in \Omega$, which is possible since $\Omega$ is bounded.

This construction is useful for the results below, for instance, in \autoref{th:CV_connex}.
\end{remark}

By definition, the multiphase function for polygons $\multiphase_\network$, defined in \autoref{eq:para}, is piecewise constant, and its corresponding superlevel sets
\begin{align}\label{eq:GammaIs}
	S_\mult = S_{(\mult_1,\ldots,\mult_m)} = \bigcap_{k=1}^m \set{\vx \in \Omega : \mult_k \network_k(\vx) > 0} = \bigcap_{k=1}^m \mathcal{S}_k^{\mult_k},
\end{align}
are polygons. The main question of interest (as a generalization of \autoref{sec:example}, where $m=1$) is the converse direction: given polygonal sets $\{P_\mult :\mult \in \{-1,1\}^m \}$, we seek to construct a family of \emph{one layer } networks $\{\network_k : k=1,\ldots,m\}$ whose zero level sets coincide with the boundaries of $\{P_\mult :\mult \in \{-1,1\}^m \}$ (see \autoref{fig:polygon_app}). Consequently, the induced multiphase function $\multiphase_\network$ is piecewise constant on this polygonal partition (compare with \autoref{eq:CV_piecewise}), that is,
\begin{equation} \label{eq:sn}
\multiphase_\network(\vx) =  \sum_{\mult \in \{-1,1\}^m} c_\mult \chi_{P_\mult}(\vx) \text{ for almost all } \vx \in \Omega\,,
\end{equation}
and $P_\mult = S_\mult$ for all $\mult \in \{-1,1\}^m$, where $S_\mult$ is as defined in \autoref{eq:GammaIs}.
This construction is formalized in \autoref{th:CV_connex} below.

\begin{theorem} \label{th:CV_connex} 	
Let $\set{\mathcal{P}_k: k=1,\ldots,m} \subseteq \Omega$ be a collection of non empty, open, connected polygons such that the boundaries of any two polygons intersect in only finitely many points, and the boundary of each polygon is a simple closed curve. For every $k = 1,\ldots, m$, define
$\mathcal{P}_k^1 := \mathcal{P}_k$ and $\mathcal{P}_k^{-1} := \Omega\backslash\mathcal{P}_k$, and for every $\mult \in \set{-1,1}^m$, define \footnote{the sets $P_\mult$ can be empty or disconnected.}
\begin{equation} \label{eq:pmult}
	P_\mult = P_{(\iota_1,\iota_2,\ldots,\iota_m)}  := \bigcap_{k=1}^m \mathcal{P}_k^{\iota_k}.
\end{equation}
Then, there exists a family of one layer  networks $\set{\network_k: k=1,\ldots,m} \subseteq \nnset$ as defined in \autoref{eq:n} such that for every $\mult \in \{-1,1\}^m$
\begin{equation}\label{eq:GammaP}
P_\mult = \bigcap_{k=1}^m \set{\vx \in \Omega : \mult_k \network_k(\vx) > 0}\;.
\end{equation}
Moreover, if $\vx \in \partial P_\mult$, then there exists some $k \in \{1,\ldots,m\}$ such that $\vx \in \{\network_k = 0\}$, and for every $k \in \{1,\ldots,m\}$ and $\vx \in \{\network_k = 0\}$, there exists  
some $\mult \in \{-1,1\}^m$ such that $\vx \in \partial P_\mult$. 
\end{theorem}

\begin{proof}
We aim to show that for every open polygon $\mathcal{P}_k$, $k=1,\ldots,m$, there exists a one layer  Heaviside network $\network_k \in \nnset$ such that

\begin{equation} \label{eq:chiPk}
\chi_{\mathcal{P}_k^{\mult_k}} = \sigma(\mult_k\network_k).
\end{equation}
Then, we have that, if $\vx \in \mathcal{P}_k^{\mult_k}$, then $\sigma(\mult_k\network_k(\vx)) = 1$ and therefore, $ \mult_k\network_k(\vx) > 0$ and conversely, if $\mult_k\network_k(\vx) > 0$, then $\chi_{\mathcal{P}_k^{\mult_k}}(\vx) = 1$ and we obtain that $\vx \in \mathcal{P}_k^{\mult_k}$. This yields the relation
$$\mathcal{P}_k^{\mult_k} = \{\vx \in \Omega: \mult_k\network_k(\vx) > 0\},$$
and \autoref{eq:GammaP} follows from the definition of $P_\mult$.
Moreover, the desired relations on the boundary of $P_\mult$ follows from the definition of $P_\mult$ and the assumption on the boundaries of $\mathcal{P}_k$.

In order to prove \autoref{eq:chiPk}, we prove that: for each $k = 1,\ldots, m$, there exists a one layer  Heaviside network $\network_k \in \nnset$ such that
    		\begin{equation} \label{eq:chi}
    		    \chi_{\mathcal{P}_k^1} = \sigma(\network_k)\;.
    		\end{equation}
Then, due to the identity $\sigma(-t) = 1 -\sigma(t)$ for all $t \in \R$, we obtain the relation
$$\chi_{\mathcal{P}_k^{-1}} = 1 - \chi_{\mathcal{P}_k} = 1 - \sigma(\network_k) = \sigma(-\network_k),$$
and together, we obtain \autoref{eq:chiPk}.

To prove \autoref{eq:chi}, we note that every edge $e_{jk}$, $j=1,\ldots,n_1$, of the polygon $\mathcal{P}_k$ is a subset of a line
    		$$\affine_{jk}(\vx) = \vw_{jk}^T(\vx - \vec{v}_{jk}) = \vw_{jk}^T\vx + b_{jk}\,,$$
    		where $\vw_j$ is the normal vector of $e_{jk}$ pointing inside the polygon $\mathcal{P}_k$, $\vec{v}_{jk}$ is the vertex, and $b_{jk} = -\vw_{jk}^T\vec{v}_{jk}$. The assumption that the edges $e_{jk}$ intersect pairwise only at finitely many points and that $\partial \mathcal{P}_k$ is a simple closed curve guarantees that each $\mathcal{P}_k$ has a connected interior region, and hence a consistent notion of inward and outward direction for its edges. Consequently, an inward unit normal vector $\vw_{jk}$ is uniquely determined for each edge $e_{jk}$.

Let
\begin{equation} \label{eq:kappa_gen}
-1 < \kappa < -1 + \frac{1}{2n_1},
\end{equation}
then the following one layer  Heaviside network (see
\autoref{eq:netshift}),
    		\begin{equation} \label{eq:net_k}
\network_k(\vx) := \sum_{j=1}^{n_1} \frac{1}{n_1}\sigma\left(\affine_{jk}(\vx)\right)+ \kappa
\end{equation}
    		satisfies \autoref{eq:chi}: Since each term $\sigma(\affine_{jk}(\vx))$ takes values in $\{0,1/2,1\},$ the sum $\sum_{j=1}^{n_1} \frac{1}{n_1}\sigma\left(\affine_{jk}(\vx)\right)$ takes values in $\{0,\frac{1}{2n_1},\frac{1}{n_1},\frac{3}{2n_1},\frac{2}{n_1},\ldots,1\}$. If all terms equal $\frac{1}{n_1}$, then the sum equals $1$. Otherwise, the maximal possible value is $1-\frac{1}{2n_1}$. As a result, for constants $\kappa$ satisfying $-1 < \kappa< -1+ \frac{1}{2n_1}$, the expression $\sum_{j=1}^m \frac{1}{n_1}\sigma(\affine_j(\vx)) + \kappa$ is strictly positive if and only if $\sigma(\affine_{jk}(\vx)) = 1$ for all $j=1,\ldots,m$, which means that
\begin{equation*}
    \chi_{P_\mult}(\vx) = \sigma\left(\network_k(\vx)\right) \text{ pointwise in } \Omega,
\end{equation*}
and we have proved \autoref{eq:chi}.
\end{proof}

\begin{remark}
We comment on the assumptions of the polygon $\mathcal{P}_k$. In a typical setting illustrated in the first row of \autoref{fig:polygon_app}, we obtain four non trivial disjoint polygons $P_\mult$. If two polygons $\mathcal{P}_1$ and $\mathcal{P}_2$ are non overlapping, we get the four intersections $\emptyset$, $\mathcal{P}_1$, $\mathcal{P}_2$ and $\Omega\backslash(\mathcal{P}_1 \cup \mathcal{P}_2)$. If $\mathcal{P}_1$ is contained within $\mathcal{P}_2$, we get the four intersections $\emptyset$, $\mathcal{P}_1$, $\mathcal{P}_2\backslash\mathcal{P}_1$ and $\Omega\backslash(\mathcal{P}_1 \cup \mathcal{P}_2)$ (see \autoref{fig:polygon_app}, second row). Polygons with holes and polygons whose boundaries are self-intersecting do not satisfy the assumptions in \autoref{th:CV_connex} (see \autoref{fig:polygon_app_not}).
\end{remark}

\begin{figure}[H]
	\centering
	\includegraphics[width=0.9\linewidth]{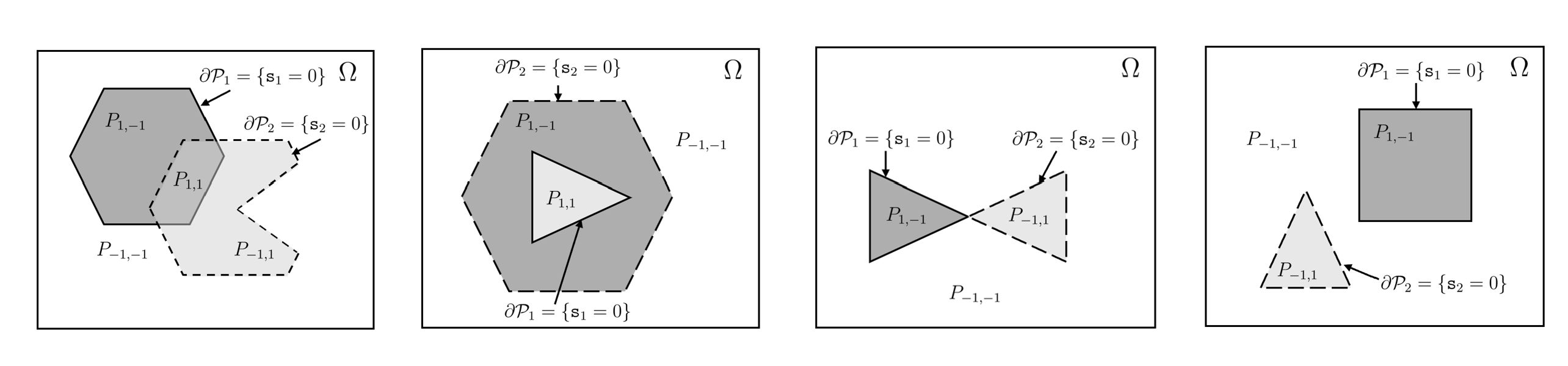}
\caption{Different configurations in which two polygons $\mathcal{P}_1, \mathcal{P}_2$ partition $\Omega$ into disjoint polygonal regions $P_{\pm1,\pm1}$, where $\mathcal{P}_1, \mathcal{P}_2$ and $P_{\pm1,\pm1}$ satisfy the assumptions in \autoref{th:CV_connex}. These four cases correspond to the polygonal approximations of the 4 configurations shown in \autoref{fig:intersection}. \textbf{First:} Typical setting in which two polygons $\mathcal{P}_1, \mathcal{P}_2$ partition $\Omega$ into four connected, disjoints polygons $P_{\pm1,\pm1}$. The non overlapping polygons $P_{\pm1,\pm1}$ are induced by a partition of $\Omega$ obtained by classifying each point in $\Omega$ according to whether it lies inside or outside each of the two polygons $\mathcal{P}_1$ and $\mathcal{P}_2$. There exist one layer  Heaviside networks $\network_1, \network_2$ (see \autoref{eq:n}) whose zero level sets coincide with the boundaries of $\mathcal{P}_1$ and $\mathcal{P}_2$. \textbf{Second to fourth:} The settings in which some of the sets $P_{\pm1,\pm1}$ maybe empty. \textbf{(Second)} $\mathcal{P}_1$ is contained in $\mathcal{P}_2$. \textbf{(Third)} $\partial\mathcal{P}_1$ and $\partial\mathcal{P}_2$ intersect in only finitely many points.  \textbf{(Fourth)} $\mathcal{P}_1$ and $\mathcal{P}_2$ are non overlapping.} \label{fig:polygon_app}
\end{figure}
\begin{figure}[H]
	\centering
	\includegraphics[width=0.4\linewidth]{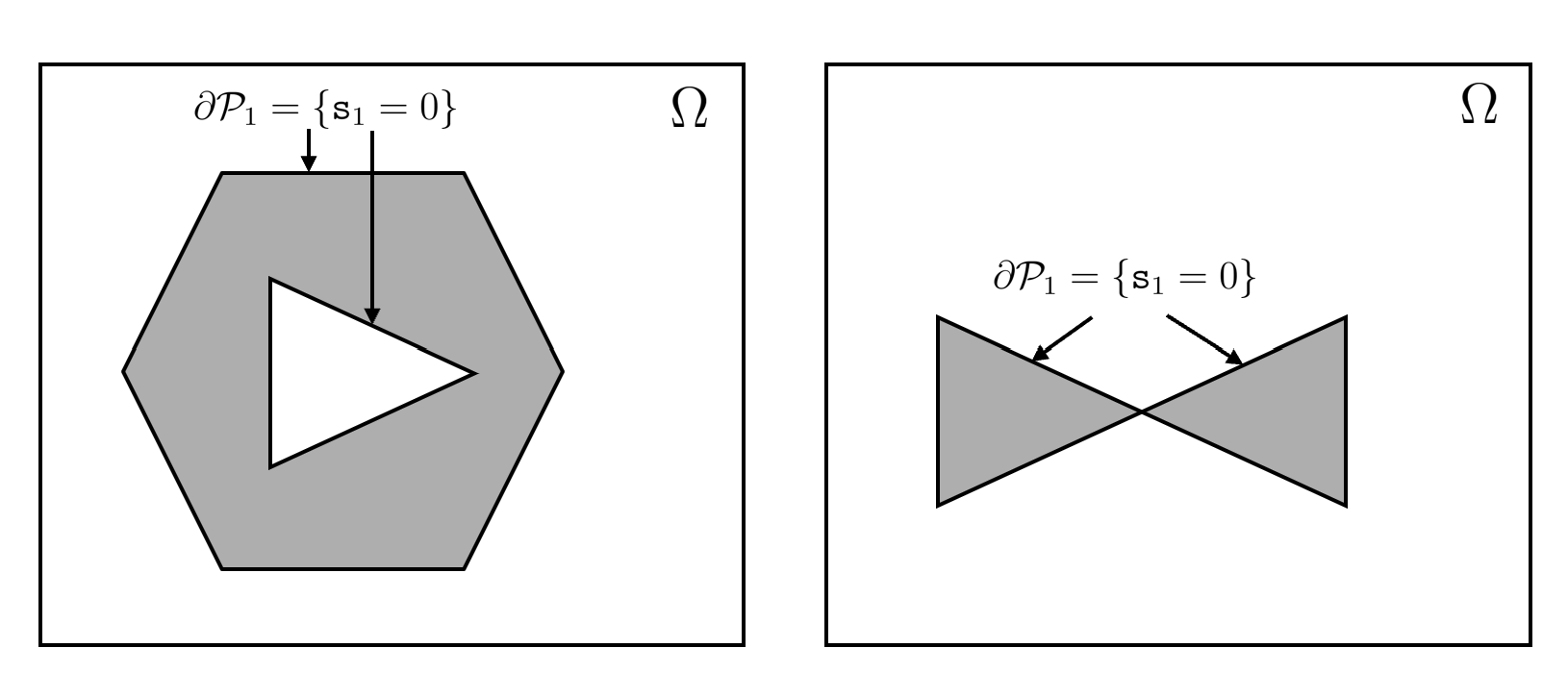}
\caption{Different configurations in which the polygons $\mathcal{P}_1, \mathcal{P}_2$ do not satisfy the assumptions in \autoref{th:CV_connex}: \textbf{(Left)} $\mathcal{P}_1$ is a polygon with holes. \textbf{(Right)} $\mathcal{P}_1$ is a polygon whose boundaries are self-intersecting. Unlike the second and third cases in \autoref{fig:polygon_app}, all the edges belong to a single polygon. In the cases of polygons with holes, each $P_\mult$ can consists of many disconnected polygonal pieces with holes. Its geometry becomes more complex and is not easily represented as an intersection of oriented half spaces. In the cases of self-intersecting boundaries, the orientation of an edge is not well-defined, and therefore the vector $\vw_{jk}$ in the proof of \autoref{th:CV_connex} is not well-defined.} \label{fig:polygon_app_not}
\end{figure}

\begin{remark}
We recall that the number of disjoint regions to be segmented is $2^m$ (some of the regions can be empty), which are obtained from $m$ level sets. The numbers of neurons $n_1 \in \N$ corresponds to the number of polygonal edges in the polygonal approximation of a single zero-superlevel set (see \autoref{ex:triangle}).

\autoref{th:CV_connex} is a generalization of the triangle example in \autoref{sec:example}, where $m=1$ and $n_1 = 3$. The condition on $\kappa$ in \autoref{eq:kappa_gen} generalizes \autoref{eq:kappa_ptw}. Moreover, the network $\network_k$ in \autoref{eq:net_k} is a generalization of the function $\kappa + \network_c$ in \autoref{eq:nettwocp}.
\end{remark}

As a consequence of \autoref{th:CV_connex}, we naturally use one layer Heaviside neural networks $\{\network_k : k=1,\ldots,m\}$ to parametrize the level set functions $\{\levelfun_k : k=1,\ldots,m\}$ of the objects to be segmented, and define the parametrized Chan-Vese functional as follows:
\begin{definition}
		The one layer  neural network parametrized Chan-Vese functional is defined by
		\begin{align} \label{eq:paraChanVese_l}
			\mathrm{CV}^{\mathrm{LS}}_{\tt f}((c_\mult)_\mult,(\vec{a}_k,{\bf W}_k,\vec{b}_k)_{k=1}^m) \nonumber
			&= \sum_{\mult \in \{-1,1\}^m}\int_{\Omega} \left( c_\mult - \f(\vx) \right)^2\prod_{k=1}^m\sigma(\mult_k\network_k(\vx)) \ \mathrm{d}\vx \\&+ \mu \sum_{k=1}^m \mathcal{H}^1(\{\network_k=0\}) + \nu \mathcal{H}^2\left(\bigcup_{k=1}^m \{\network_k > 0\}\right).
		\end{align}
	\end{definition}
	The one layer  neural network parametrized Chan-Vese functional can be used for \emph{multi-polygonal} segmentation: Instead of minimizing $\mathrm{CV}^{\mathrm{LS}}_{\tt f}$ from \autoref{eq:CV_LS} over arbitrary level set functions to determine the segments of the function ${\tt f}$, we constrain the optimization of the functional to piecewise constant functions on polyhedral partitions of the domain.

Moreover, every piecewise constant functions on polyhedral sets $\multiphase_\network$ as defined in \autoref{eq:para} coincides pointwise with a two layer Heaviside neural network $\nettwo \in \nntwo$ as defined in \autoref{eq:t}.

\begin{lemma} \label{thrm:lsnn}
	Let $m \in \N$ be fixed. Then, every multiphase Heaviside function for polygons $\multiphase_\network$, defined in \autoref{eq:para}, can be represented as a two layer Heaviside network
	\begin{equation} \label{eq:lsnn}
		\begin{aligned}
      \vx \in \Omega \subseteq \R^2 \mapsto \multiphase_\network(\vx) =
			\sum_{\mult \in \{-1,1\}^m} c_\mult \sigma\left(\kappa + \sum_{j=1}^{n_1}\frac{1}{n_1} \sigma(\mult_j \affine_j(\vx))\right) \in \nntwo\;,
		\end{aligned}
	\end{equation}
    where $\kappa \in \R$ is a constant satisfying \autoref{eq:kappa_gen}.
\end{lemma}
\begin{proof}
For every multiphase Heaviside function for polygons $\multiphase_\network$, we have the relation (see  \autoref{eq:sn}):
\begin{equation}
\multiphase_\network(\vx) =  \sum_{\mult \in \{-1,1\}^m} c_\mult \chi_{P_\mult}(\vx) \text{ for almost all } \vx \in \Omega\,,
\end{equation}
where $\{P_\mult : \mult \in \{-1,1\}^m \}$ are polygons. Therefore, to obtain \autoref{eq:lsnn}, we show that 
\begin{equation}\label{eq:1layer-segment}
    \chi_{P_\mult}(\vx) = \sigma\left(\kappa + \sum_{j=1}^{n_1} \frac{1}{n_1} \sigma(\mult_j \affine_j(\vx))\right) \in \nntwo \text{ pointwise in }\Omega.
\end{equation}

Since $\{P_\mult : \mult \in \{-1,1\}^m \}$ are polygons, we apply the same argument as in the proof of \autoref{th:CV_connex} to establish \autoref{eq:1layer-segment}: For constants $\kappa$ satisfying $-1 < \kappa < -1 + \frac{1}{2n_1}$, the expression $\sum_{j=1}^{n_1} \frac{1}{n_1} \sigma(\mult_j \affine_j(\vx)) + \kappa$ is positive if and only if $\sigma(\mult_j \affine_j(\vx)) = 1$ for all $j=1,\ldots,n_1$, yielding \autoref{eq:1layer-segment}.

\end{proof}
We note that in \autoref{eq:lsnn}, some parameters are fixed ($a_{ij}=\frac{1}{n_1}$, $d_{ij}=\kappa$). This result generalize \autoref{eq:CVfctpara} in \autoref{sec:example} where $m=1$ and $n_1 = 3$. Moreover, \autoref{eq:1layer-segment} corresponds to a customized two layer network $\nettwo_c$ as defined in \autoref{eq:d}, which, by construction, is equivalent to applying a Heaviside activation to a one layer  network output shifted by a constant (see \autoref{de:cust}).

\subsection{Regularized Heaviside networks} \label{sec:approx_}\hspace{0pt}\\
For numerical implementation, the Heaviside function in the parametrized Chan-Vese functional $\mathrm{CV}^{\mathrm{LS}}_{\tt f}$ (see \autoref{eq:para}) is approximated by a sigmoid function $\sigma_\ve$. With this approximation, we obtain the functional $\mathrm{CV}^{\mathrm{LS}}_{\tt f, \ve}$. Accordingly, if we use $\sigma_\ve$ instead of $\sigma$ in the definition of one layer  networks $\network$ in $\nnset$ (see \autoref{eq:n}) and $\nettwo$ in  $\nntwo$ (see \autoref{eq:t}), we get $\network_\ve$ in $\nnset_\ve$ and $\nettwo_\ve$ in $\nntwo_\ve$, respectively. Analogously, we obtain the smooth level set multiphase function $\multiphase_{\levelfun,\ve}$ and the smooth multiphase function for polygons $\multiphase_{\network,\ve}$ from $\multiphase_\levelfun$ (see \autoref{eq:multiphaselevelset}) and $\multiphase_{\network}$ (see \autoref{eq:para}), respectively. The approximation is justified by the following lemma:

\begin{lemma}[Properties of $\sigma_\ve$] \label{le:sigmoidprop}
	For all $\ve > 0$, the sigmoid function $\sigma_\ve$ is Lipschitz continuous, and the family $\{ \sigma_{\ve}: \ve>0\}$ approximates the Heaviside function $\sigma$ in $L^1(\R), L^2(\R)$, and pointwise in $\R$.
\end{lemma}

Together, these findings establish a structural link between the Chan-Vese model and neural networks (see \autoref{fig:multiphase_parametrized}). In particular, the classical level set formulation can be expressed in a neural network approximation, facilitating numerical implementation using established architectures and optimization algorithms. The analysis also identifies that a minimal depth of two layers is sufficient for image segmentation.

\begin{figure}[H]
	\centering
	\includegraphics[width=0.9\linewidth]{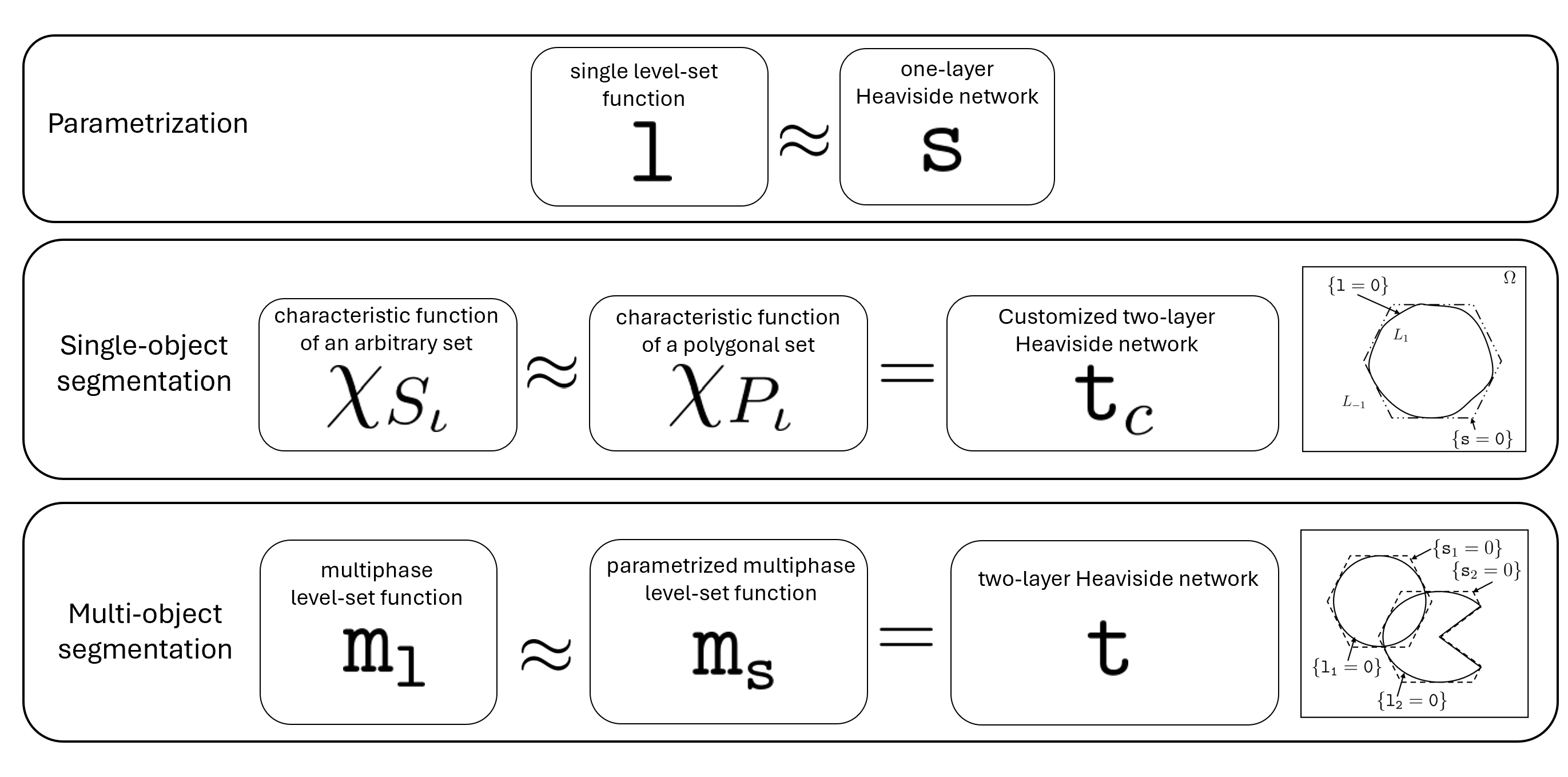}
	\caption{The Chan–Vese segmentation algorithm consists in finding the zero level set of a function that can be parametrized by a one layer Heaviside network (see \autoref{eq:para}). This is equivalent to approximating the segmented region by a polygonal set (see \autoref{th:CV_connex}). Moreover, since polygonal sets can be represented by a two layer network (see \autoref{thrm:lsnn}), we obtain an approximation result for general sets: the characteristic function of the to-be-segmented region $L_\mult$ (\autoref{eq:GammaI}) can be approximated by a customized two layer Heaviside network $\nettwo_c$ (\autoref{eq:d}), while the multiphase segmentation function $\multiphase_\levelfun$ (\autoref{eq:multiphaselevelset}) can be approximated by a two layer Heaviside network $\nettwo$ (\autoref{eq:t}).}\label{fig:multiphase_parametrized}
\end{figure}

\section{Algorithm} \label{subsec:alg}\hspace{0pt}\\
In order to minimize the one layer  network parametrized Chan–Vese functional, defined in \autoref{eq:paraChanVese_l}, which is highly non convex with respect to the parameters $(c_\mult),(\vec{a}_k,{\bf W}_k,\vec{b}_k)_{k=1}^m$ an appropriate initialization of the network parameters is required. This initialization is obtained through a learning process, which consists of two steps:
\begin{itemize}
    \item \textbf{Step 1:} Train a neural network to generate an initial network structure for parametrization $(c_\mult),(\vec{a}_k,{\bf W}_k,\vec{b}_k)_{k=1}^m$ (\autoref{alg:nn-train}). The basic learning framework is summarized in \autoref{fig:train}. The training, validation and testing procedure is summarized in \autoref{alg:nn-train}.
    \item \textbf{Step 2:} After training, the optimized network parameters initialize the evolution of the level set functions, which is driven by stochastic optimization of the network parameters to minimize the parametrized Chan–Vese functional $\mathrm{CV}^{\mathrm{LS}}_{\tt f,\ve}$ (see \autoref{alg:paraCV}).
\end{itemize}
\begin{figure}[H]
    \centering
    \includegraphics[width=0.7\linewidth]{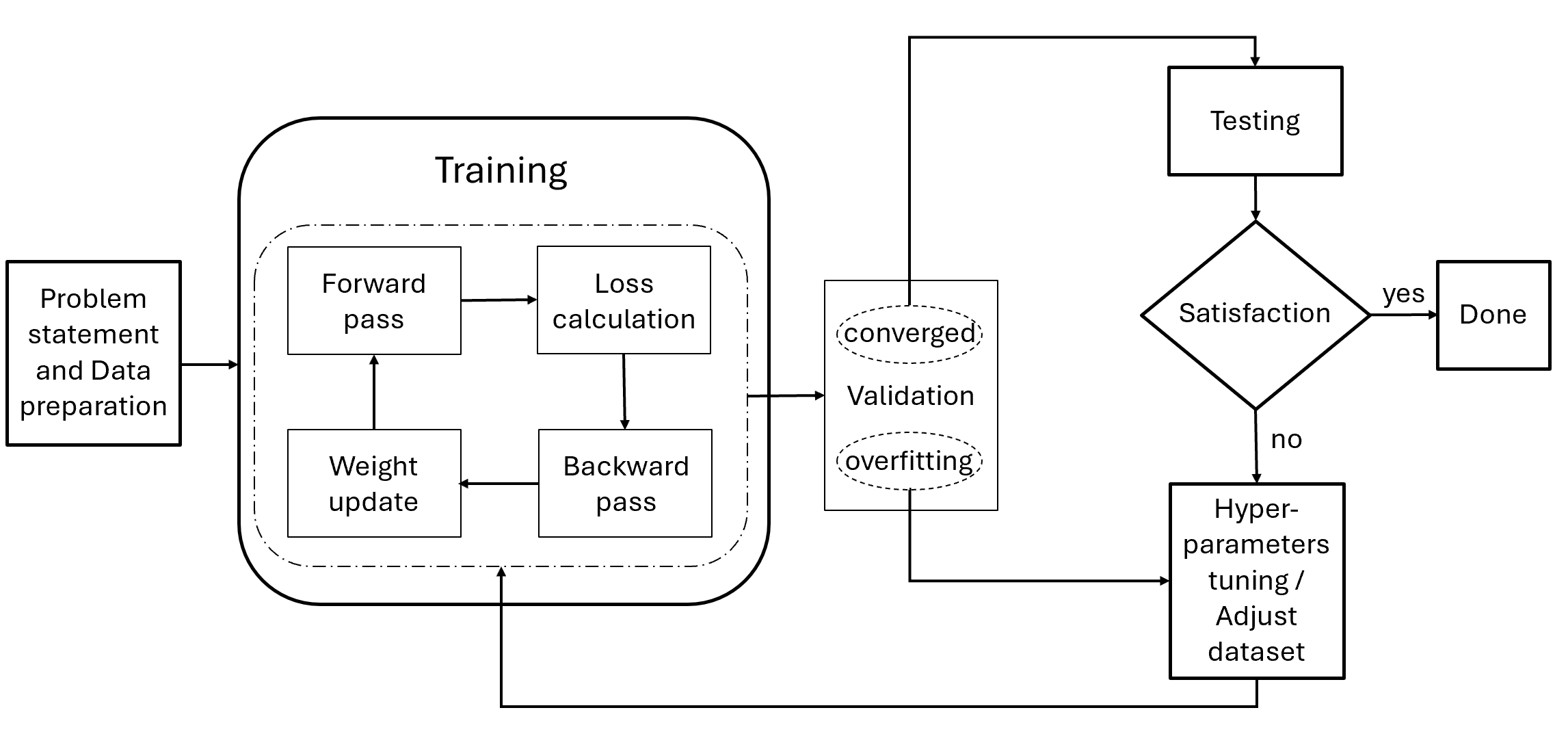}
    \caption{Flowchart of the prior learning process.}\label{fig:train}
\end{figure}

\begin{algorithm}[h!]
\caption{Training 1-Layer Neural Network for Multiphase Segmentation}
\label{alg:nn-train}
\KwIn{
    Training set $\mathcal{F} = \{{\tt f}^i_{\mathrm{input}}\}$,
    (Optional) Targeted set $\mathcal{T} = \{{\tt f}^i_{\mathrm{target}}\}$,
    Validation set $\mathcal{G} = \{{\tt g}^i\}$,
    $m$ phases, $n_1$ neurons,
    max epochs $E$, patience $P$,
    (Optional) Supervised loss $\mathcal{L}$, Optimizer
}
\KwOut{Trained parameters $\Tilde{\Theta}$}

\SetAlgoLined

\textbf{// Initialize $m$ networks} \\
\For{$k = 1$ \KwTo $m$}{
    $\vec{a}_k \in \R^{n_1}, {\bf W}_k \in \R^{n_1 \times 2}, \vec{b}_k \in \R^{n_1}$ random initialization (e.g., $\mathcal{N}(0,0.01)$) \;
    $\Theta_k \gets (\vec{a}_k, {\bf W}_k, \vec{b}_k)$ \;
}
$\Tilde{\Theta} \gets (\Theta_1, \dots, \Theta_m)$ \\

\BlankLine
\textbf{// Training Loop} \\
$best\_loss \gets \infty$, $wait \gets 0$ \\
\For{$epoch = 1$ \KwTo $E$}{
    \For{each image ${\tt f}^i_{\mathrm{input}} \in \mathcal{F}_N$}{
        \tcp{Forward}
		\For{$k=1$ \KwTo $m$}{
        Compute network outputs $\network_\ve[\Theta_k](\vx) = \sum_{j=1}^{n_1} (\vec{a}_k)_j \sigma (({\bf W}_k)_j^T\vx + (\vec{b}_k)_j)$ \\
		}
        Compute region means $c_\mult^\ve(\tilde{\Theta}) = \frac{\int_\Omega {\tt f}(\vx)\prod_{k=1}^m \sigma_\ve(\mult_k\network_\ve[\Theta_k](\vx)) \mathrm{d}\vx}{\int_\Omega \prod_{k=1}^m \sigma_\ve(\mult_k\network_\ve[\Theta_k](\vx)) \mathrm{d}\vx} \text{ for } \mult \in \{-1,1\}^m.$\\
        \tcp{Loss}
        $L \gets E_{CV}^{{\tt f}_i,\ve}(\{c_\mult^\ve(\Tilde{\Theta})\}_{ \mult \in \{-1,1\}^m},\{\network_\ve[\Theta_k]\}_{j=1}^m)$ \quad $(+ \text{ optional supervised loss } \mathcal{L}({\tt f}^i_{\mathrm{input}} - {\tt f}^i_{\mathrm{target}}))$ \\
        \tcp{Backward}
        Compute the gradient $\nabla_{\Tilde{\Theta}} L$\\
        \tcp{Update}
        $\Tilde{\Theta} \gets$ Optimizer$(\Tilde{\Theta}, \nabla_{\Tilde{\Theta}} L)$
    }

    \tcp{Validation}
    $val\_loss \gets \frac{1}{\abs{\mathcal{G}}} \sum_{{\tt g}_i \in \mathcal{G}} E_{CV}^{{\tt g}_i,\ve}(\{c_\mult^\ve(\Tilde{\Theta})\}_{ \mult \in \{-1,1\}^m},\{\network_\ve[\Theta_k]\}_{j=1}^m)$\\
    \If{$val\_loss < best\_loss$}{
        $best\_loss \gets val\_loss$, \quad $wait \gets 0$ \\
        $\Tilde{\Theta}_{\text{best}} \gets \Tilde{\Theta}$ \\
    }\Else{
        $wait \gets wait + 1$ \\
        \If{$wait \geq P$}{
            \textbf{break} \tcp{Early stopping}
        }
    }
}
\Return $\Tilde{\Theta}_{\text{best}} =: \overline{\Theta} = (\overline{\Theta}_1, \cdots, \overline{\Theta}_m)$

\BlankLine
\tcp{Testing}
\For{new image ${\tt g}$}{
    Compute network outputs $\network_\ve[\overline{\Theta}_k]$\\
    Compute region means $c_\mult^\ve(\overline{\Theta})$\\
    Compute multiphase representation using $\Tilde{\Theta}_{\text{best}}$ and region means $\multiphase = \sum_{\mult \in \{-1,1\}^m} c_\mult^\ve(\overline{\Theta}) \prod_{k=1}^m \sigma_\ve(\mult_k\network_\ve[\overline{\Theta}_k])$\
}
\end{algorithm}

\begin{algorithm}[H]
\caption{Multiphase Parametrized Chan-Vese Segmentation with Mini-Batch SGD}
\label{alg:paraCV}
\KwIn{
    Image $\tt f$, number of phases $m \geq 1$, (optional) initial parameters $\{\Theta_k^0 = (\vec{a}_k^0, {\bf W}_k^0, \vec{b}_k^0)\}_{k=1}^m$,
    regularization constant $\mu > 0$, updating rate $\eta$, max iterations $T$, batch size $B$, smoothing parameter $\ve > 0$, tolerance $\tau > 0$
}

\KwOut{Optimized parameters $\{\Theta_k^*\}_{k=1}^m$, region means $\{c_\mult^*\}_{\mult \in \{-1,1\}^m}$, multiphase Chan-Vese function $\multiphase$}

\SetAlgoLined
\SetKwFunction{UpdateMeans}{UpdateRegionMeans}
\SetKwFunction{UpdateFunction}{ParametrizedLevelsetFunction}
\SetKwFunction{RandomInit}{RandomInitialParameters}

\BlankLine
\textbf{// Initialization:}\\
\For{$k = 1$ \KwTo $m$}{
    \If{$\Theta_k^0$ is not provided}{
       $\vec{a}_k^0 \in \R^{n_1}, {\bf W}_k^0 \in \R^{n_1 \times 2}, \vec{b}_k^0 \in \R^{n_1}$ random initialization (e.g., $\mathcal{N}(0,0.01)$)\\
 $\Theta_k \gets (\vec{a}_k, {\bf W}_k, \vec{b}_k)$\; }\Else{
        $\Theta_k \gets \Theta_k^0$\;
    }
 Compute  $\network_\ve[\Theta_k](\vx) = \sum_{j=1}^{n_1} (\vec{a}_k)_j \sigma (({\bf W}_k)_j^T\vx + (\vec{b}_k)_j)$ \\
}
Calculate regions mean: $c_\mult^\ve = \frac{\int_\Omega {\tt f}(\vx)\prod_{k=1}^m \sigma_\ve(\mult_k\network_\ve[\Theta_k](\vx)) \mathrm{d}\vx}{\int_\Omega \prod_{k=1}^m \sigma_\ve(\mult_k\network_\ve[\Theta_k](\vx)) \mathrm{d}\vx} \text{ for } \mult \in \{-1,1\}^m.$

\BlankLine
\textbf{// Mini-batch SGD optimization} \\
\For{$t = 1$ \KwTo $T$}{
    $\mathcal{B} \gets$ random batch of $B$ pixels\;
    \For{$k = 1$ \KwTo $m$}{
        $\mathbf{g}_k \gets \mathbf{0}$\;

        \For{$\vx \in \mathcal{B}$}{
        Compute  $\network_\ve[\Theta_k](\vx) = \sum_{j=1}^{n_1} (\vec{a}_k)_j \sigma (({\bf W}_k)_j^T\vx + (\vec{b}_k)_j)$ \\
		 $\mathbf{g}_k \gets \mathbf{g}_k + \nabla_{\Theta_k} E^{{\tt f},\ve}_{CV}(\{c_\mult^\ve\}_{ \mult \in \{-1,1\}^m},\{\network_\ve[\Theta_k]\}_{j=1}^m)$
        }

    $\mathbf{g}_k \gets \mathbf{g}_k / B$\;
    $\Theta_k \gets \Theta_k - \eta \mathbf{g}_k$\;
    }

   Update region means $c_\mult^\ve = \frac{\int_\Omega {\tt f}(\vx)\prod_{k=1}^m \sigma_\ve(\mult_k\network_\ve[\Theta_k](\vx)) \mathrm{d}\vx}{\int_\Omega \prod_{k=1}^m \sigma_\ve(\mult_k\network_\ve[\Theta_k](\vx)) \mathrm{d}\vx} \text{ for } \mult \in \{-1,1\}^m.$\\

    \BlankLine
    \textbf{// Early stopping} \\
    \If{$\max_k \|\mathbf{g}_k\| < \tau$}{\textbf{break}}
}

\textbf{// Final segmentation} \\
$\multiphase = \sum_{\mult \in \{-1,1\}^m} c_\mult \prod_{k=1}^m \sigma_\ve(\mult_k\levelfun_k)$\;
\Return $\{\Theta_k\}$, $\{c_\mult\}$, $\multiphase$\;
\end{algorithm}

\section{Numerical results} \label{numerics}
To verify our theoretical results, we run the parametrized Chan–Vese algorithm with $m=4$ (allowing at most $2^4=16$ distinct segments) using both randomly initialized parameters and pre-trained parameters for initialization, and then compare the results.

Our algorithm (\autoref{alg:paraCV}) requires an initialization of the level set functions, which is computed from training samples. For training the initialization, we used \autoref{alg:nn-train}. The datasets for training consist of $2000$ images of size $50 \times 50$ pixels, each containing three random objects chosen from circles, triangles, squares, hearts, and rectangles, with varying sizes and positions. The dataset is split into training and validation sets using an 80–20 ratio. Each network, corresponding to a level set function, is a linear neural network consisting of a single hidden layer with $10 \times 50 \times 50$ neurons. It uses the sigmoid activation function with $\ve=1$ and is trained using the AdamW optimizer with a learning rate of $3.10^{-5}$ and a weight decay rate of $3.10^{-3}$. The model is trained for 10 epochs, where each epoch represents one full pass through the training dataset. The hyperparameters were chosen through iterative testing to best match the characteristics of our dataset and achieve stable training performance.

For parametrized Chan-Vese functional optimization (\autoref{alg:paraCV}), the regularization parameters are chosen as $\mu = 0.5, \nu = 0$. Smooth approximation of Heaviside and Dirac-delta functions are used with $\ve = 1$. The gradient of the Chan-Vese functional with respect to the parameters is calculated using automatic differentiation (autograd), which relies on the analytic gradient calculation using the chain rule. The parameters are then optimized using AdamW.

\paragraph{Visualization results}
We evaluate the performance of our proposed method qualitatively through visual inspection of the segmentation boundaries and compare the convergence behavior of the Chan–Vese functional across different conditions: using trained versus randomly initialized networks, varying $\ve$ values, and employing stochastic versus deterministic optimizers.

From \autoref{fig:result_10ite}, we observe that after 10 iterations, the object boundaries are detected in both the pre-trained and randomly initialized cases.  Additional examples are presented in \autoref{fig:result}. These examples are more complex and require 20 iterations to recover the boundaries. Moreover, we observe that the boundaries obtained with random initialization are less clean and contain more spurious pixels compared to those obtained with the pre-trained model. Some of these irregularities can be attributed to the inherent stochasticity of mini-batch optimizers. By operating on subsets of the data, these optimizers reduce memory and computational cost per iteration but introduce variability in the optimization path. This can result in isolated noisy pixels and minor differences in the final solution, even when starting from identical initializations and datasets, leading to slightly less consistent outcomes.

There is an inherent trade-off between precision and stability on one hand and speed and scalability on the other. Stochastic optimizers generally converge faster and require less memory per iteration compared to deterministic methods, which typically rely on full-gradient information. They are also less sensitive to initialization. In contrast, quasi-Newton methods, such as L-BFGS, can diverge depending on the starting point, making random initialization suboptimal. Overall, stochastic optimizers are faster, whereas L-BFGS with a good initialization tends to achieve higher precision. The choice of optimizer should be guided by the desired balance between computational efficiency and accuracy.

Furthermore, from the plots of the parametrized Chan-Vese energy functional in \autoref{fig:result_10ite}, we observe that the Chan-Vese energy value for the pre-trained initialization is lower than that of the randomized initialization, indicating that the pre-trained contours are initialized closer to the object.

We then examine the performance of the algorithm on noisy input data by adding random Gaussian noise with level $0.05$ to the original noise-free image. From \autoref{fig:result_noisy}, we observe that the algorithm remains capable of detecting the boundary of the object even in the presence of noise.

\begin{figure}[H]
    \centering
	     \includegraphics[width=0.15\linewidth]{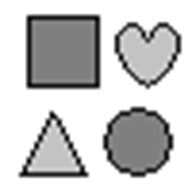}
    \includegraphics[width=0.15\linewidth]{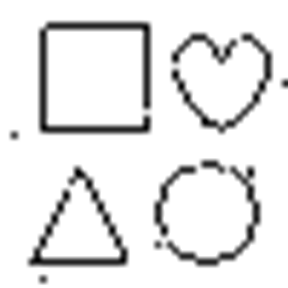}
    \includegraphics[width=0.15\linewidth]{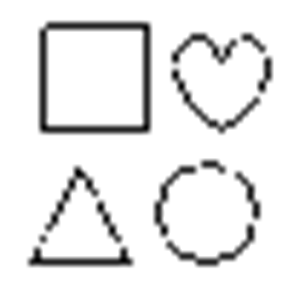}\\
 \includegraphics[width=0.5\linewidth]{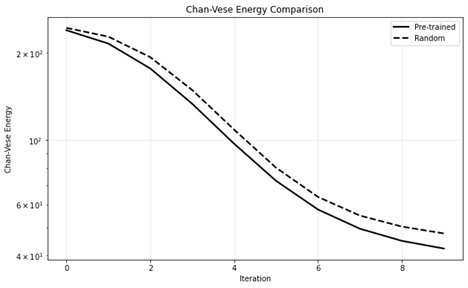}\\
        \includegraphics[width=0.15\linewidth]{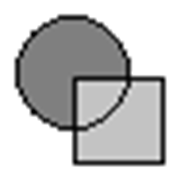}
    \includegraphics[width=0.15\linewidth]{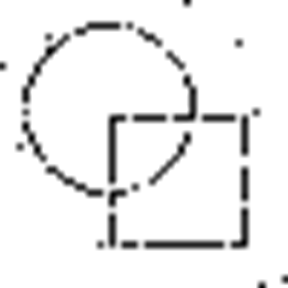}
    \includegraphics[width=0.15\linewidth]{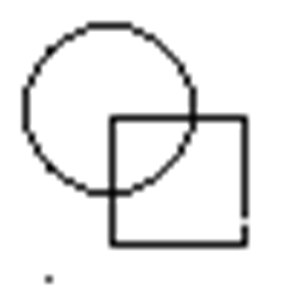}\\
    \includegraphics[width=0.5\linewidth]{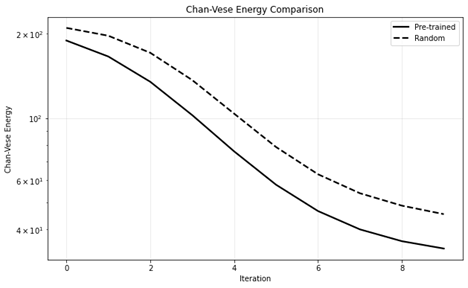}\\

    \caption{Input image $\tt f$ to be segmented (left) and segmentation results of the parametrized Chan-Vese method after 10 iterations using randomized initialization (middle) and pre-trained initialization (right). The plot shows the corresponding parametrized Chan-Vese energy functional for random initialization (dashed line) and pre-trained initialization (solid line).}
    \label{fig:result_10ite}
\end{figure}

\begin{figure}[H]
    \centering
    \includegraphics[width=0.15\linewidth]{img/numeric_new/Test2.png}
    \includegraphics[width=0.15\linewidth]{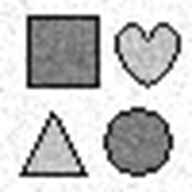}
    \includegraphics[width=0.15\linewidth]{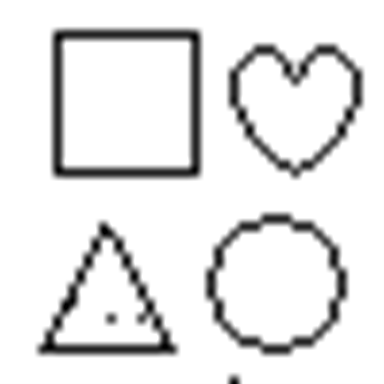}
    \includegraphics[width=0.15\linewidth]{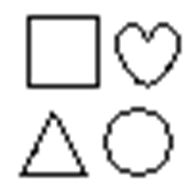}\\
\includegraphics[width=0.15\linewidth]{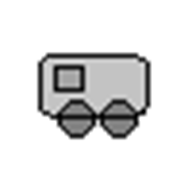}
    \includegraphics[width=0.15\linewidth]{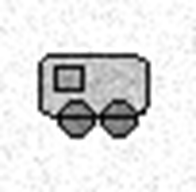}
    \includegraphics[width=0.15\linewidth]{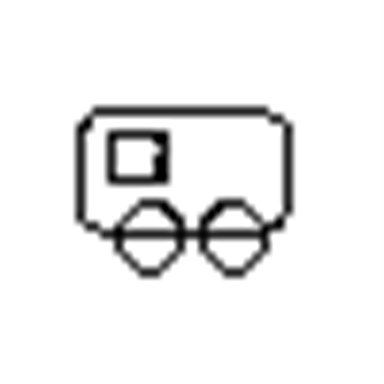}
    \includegraphics[width=0.15\linewidth]{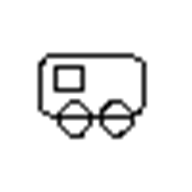}\\
    \caption{Noise-free original image (first), noisy input image $\tt f$ to be segmented with noise level $0.05$ (second) and segmentation results of the parametrized Chan-Vese method after 50 iterations using randomized initialization (third) and pre-trained initialization (fourth).}
    \label{fig:result_noisy}
\end{figure}

\begin{figure}[H]
    \centering
    \includegraphics[width=0.15\linewidth]{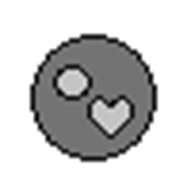}
    \includegraphics[width=0.15\linewidth]{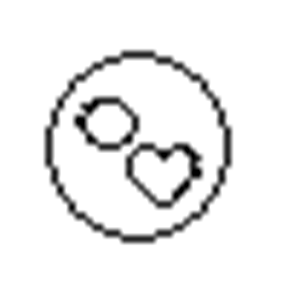}
    \includegraphics[width=0.15\linewidth]{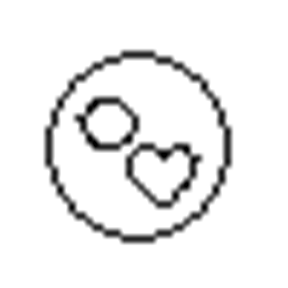}\\
    \includegraphics[width=0.15\linewidth]{img/numeric_new/Test4.png}
    \includegraphics[width=0.15\linewidth]{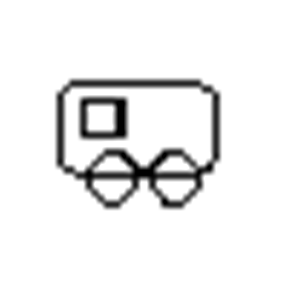}
    \includegraphics[width=0.15\linewidth]{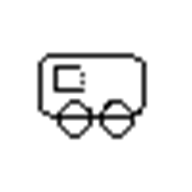}\\
    \includegraphics[width=0.15\linewidth]{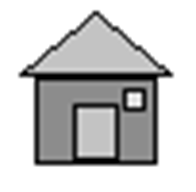}
    \includegraphics[width=0.15\linewidth]{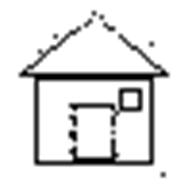}
    \includegraphics[width=0.15\linewidth]{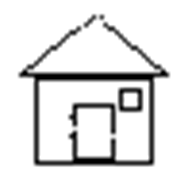}\\
  \includegraphics[width=0.15\linewidth]{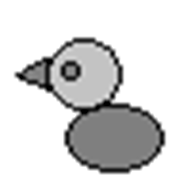}
    \includegraphics[width=0.15\linewidth]{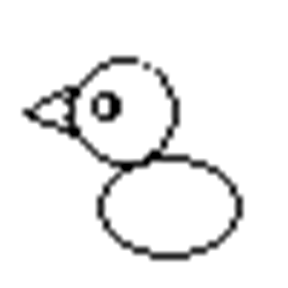}
    \includegraphics[width=0.15\linewidth]{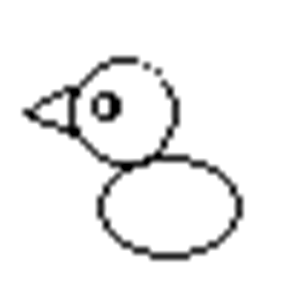}\\
    \caption{Additional results for the parametrized Chan–Vese algorithm. Input image $\tt f$ to be segmented (left) and segmentation results of the parametrized Chan-Vese method after 20 iterations using randomized initialization (middle) and pre-trained initialization (right).}
    \label{fig:result}
\end{figure}

Finally, the effect of the parameter $\mu$ in \autoref{eq:CV_LS} is illustrated in \autoref{fig:mu}. The energy functional converges similarly across different values of $\mu$ and the results are largely insensitive to the choice of $\mu$, with some small differences, such as a few spurious pixels appearing. Visual inspection indicates that $\mu = 1$ (the regularization parameter in \autoref{eq:CV_og}) produces the clearest results.
\begin{figure}[h]
    \centering
    \includegraphics[width=0.5\linewidth]{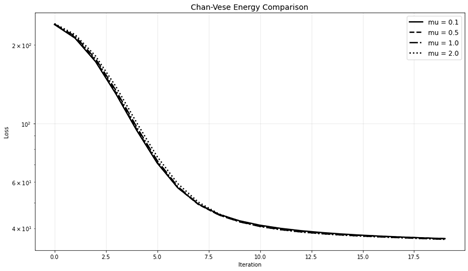}\\
    \includegraphics[width=0.15\linewidth]{img/numeric_new/Test2.png} \quad
    \includegraphics[width=0.15\linewidth]{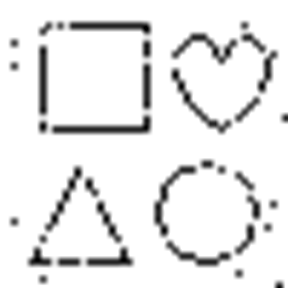}
    \includegraphics[width=0.15\linewidth]{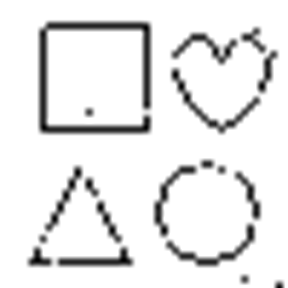}
    \includegraphics[width=0.15\linewidth]{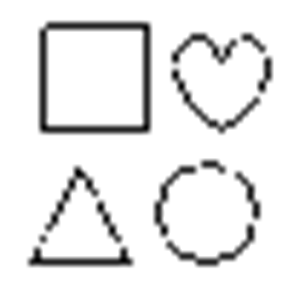}
    \includegraphics[width=0.15\linewidth]{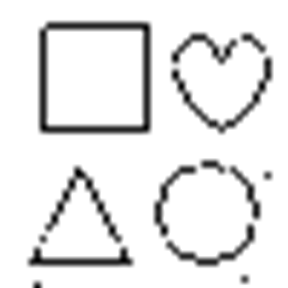}
    \caption{First row: Parametrized Chan-Vese energy functional for $\mu=0.1, 0.5, 1, 2$. Second row, from left to right: Input image and results of parametrized Chan-Vese algorithm after 20 iterations for $\mu =0.1, 0.5, 1, 2$.}
    \label{fig:mu}
\end{figure}

\section{Conclusion}
In this paper, we provide insight into network architectural design for image segmentation by showing that deep architectures with two hidden layers are particularly well suited for representing polygonal approximations of level set based segmentations. This establishes a direct link between the classical variational Chan–Vese model and neural network representations, and provides an interpretable framework for neural networks in image segmentation tasks.

Moreover, we propose a data driven approach to the Chan–Vese method for image segmentation and multiple-object segmentation, in which level set representations are parametrized by neural networks. We learn geometric structures directly from data and use them to initialize the parametrized Chan–Vese evolution. This leads to a geometry aware initialization strategy that improves convergence, particularly when test images share structural similarities with the training data.

\subsection*{Acknowledgements}
This research was funded in whole, or in part, by the Austrian Science Fund
(FWF) 10.55776/P34981 (OS) -- New Inverse Problems of Super-Resolved Microscopy (NIPSUM),
SFB 10.55776/F68 (OS) ``Tomography Across the Scales'', project F6807-N36
(Tomography with Uncertainties), 10.55776/T1160 (CS) ``Photoacoustic Tomography: Analysis and Numerics'', National Natural Science Foundation of China (Grant No. 41030040 (CS)). For open access purposes, the author has applied a CC BY public copyright license to any author-accepted manuscript version arising from this submission.
The financial support by the Austrian Federal Ministry for Digital and Economic
Affairs, the National Foundation for Research, Technology and Development and the Christian Doppler
Research Association is gratefully acknowledged.

\section*{References}
\renewcommand{\i}{\ii}
\printbibliography[heading=none]

\end{document}